\documentclass{amsart}
\usepackage[all,hyperref]{bi-discrete}
\usepackage[backend=biber,maxbibnames=5,maxalphanames=5,style=alphabetic,bibencoding=utf8,giveninits,url=false,isbn=false]{biblatex}
\AtBeginBibliography{\small}

\renewcommand{\dashint}{\fint}
\usepackage{tikz}
\usepackage{esint}

\begin{document}
	
\title{Double phase meets Muckenhoupt}%

\author{Daviti Adamadze}
 \address{Daviti Adamadze, University Bielefeld, Universit\"atsstrasse 25, 33615 Bielefeld, Germany}
\email{daviti.adamadze@uni-bielefeld.de}

\author{Lars Diening}
\address{Lars Diening, University Bielefeld, Universit\"atsstrasse 25, 33615 Bielefeld, Germany}
\email{lars.diening@uni-bielefeld.de}

\author{Tengiz Kopaliani}
 \address{Faculty of Exact and Natural Sciences, Javakhishvili Tbilisi State
University, 13, University St., Tbilisi, 0143, Georgia}
\email{tengiz.kopaliani@tsu.ge}

\author{Jihoon Ok}
\address{Jihoon Ok, Department of Mathematics, Sogang University, Seoul 04107, Republic of Korea}
\email{jihoonok@sogang.ac.kr}

\thanks{Daviti Adamadze and Lars Diening gratefully acknowledge the financial support by the Deutsche Forschungsgemeinschaft (DFG, German Research Foundation) IRTG 2235 - Project number 282638148 and SFB 1283/2 2021 – 317210226, respectively. Tengiz Kopaliani was supported by Shota Rustaveli National Science Foundation of Georgia FR-22 17770. Jihoon Ok was supported by the National Research Foundation of Korea by the Korean Government (NRF-2022R1C1C1004523).  }%


\begin{abstract}
	In this paper we generalize the  famous result of~\textcite{FKS} to the double phase model. In particular, 
  we  work with minimal assumptions on the modulating coefficient by introducing a Muckenhoupt-type condition on generalized Orlicz spaces. We develop a complete theory equivalent to that of classical Muckenhoupt weights, including the boundedness of the maximal operator and Sobolev-\Poincare{} estimates. We combine this with the De~Giorgi technique to show H\"older continuity of the solutions.
\end{abstract}

\maketitle

\tableofcontents


\section{Introduction}
\label{sec:introduction}

In this paper we consider the well-known double phase model under minimal assumptions on the modulating coefficient. The double phase energy is given by
\begin{align}
  \label{eq:doublephase-energy-intro}
  \mathcal{J}(v) := \int_\Omega \tfrac 1p \abs{\nabla v(x)}^p+ \tfrac 1q a(x)\abs{\nabla v(x)}^q\,dx,
\end{align}
where $1<p<q<\infty$ and $a \,:\, \Omega \to [0,\infty)$ is the modulating coefficient. We are interested in the H\"older regularity of local minimizers. These satisfy the Euler-Lagrange equation
\begin{align}
 	\label{eq:pde}
	-\divergence\big(\abs{Du}^{p-2}Du + a(x)\abs{Du}^{q-2}Du\big)&= 0 
	\quad \text{in }\ \Omega.
\end{align}
The double phase model was introduced by Zhikov in~\cite[Example~3.1]{Zhik5} in the context of the Lavrentiev phenomenon. He showed that if the modulating coefficient is too bad, then minimizers over Sobolev functions and smooth functions may differ. This example was later refined in~\cite{ELM,BDS}.

However, if the modulating coefficient $a(x)$ is $C^{0,\alpha}$-H\"older continuous, then the Lavrentiev gap can vanish and solutions can be regular.  In fact, it has been shown by Baroni, Colombo and Mingione~\cite{BCMregularity,CoMi1}  that the solutions become $C^{1,\beta}$-H\"older continuous for some~$\beta>0$ provided that
\begin{align}
  \label{eq:gap-condition}
  \frac qp \le 1 + \frac \alpha n.
\end{align}
This condition \eqref{eq:gap-condition} can be relaxed if one already has some a~priori knowledge on additional regularity of~$u$ like boundedness, \cite{CoMi2,BCMregularity,HastoOk}. All of these results also extend to the weak concept of quasi-minimizers~\cite{Ok1}.  The sharpness of these conditions for general H\"older continuous~$a(x)$ follows from the sophisticated examples in~\cite{BDS,BalciDieningSurnachev2025} .
 In \cite{BalciDieningSurnachev2025}, they show that if~\eqref{eq:gap-condition} fails, then there exists a $C^{0,\alpha}$-modulating coefficient~$a(x)$ with a solution~$u$ that is not even in~$W^{1,p+\epsilon}$ for any~$\epsilon>0$. This improves a related result in~\cite{FonsecaMalyMingione2004}. Let us mention that in Example~\ref{exa:power} we show that the specific modulating coefficients~$a(x) = \min \set{\abs{x}^\alpha,1}$ with $\alpha>-n$ always ensure H\"older continuity of the solution although~\eqref{eq:gap-condition} fails for all~$\alpha \in (-n,n(\frac{q}{p}-1))$.

Up to now, all results on regularity of the solution to \eqref{eq:pde} require at least $C^{0,\alpha}$-H\"older continuity of the modulating coefficient together with  \eqref{eq:gap-condition}. Surprisingly, without this assumption, no regularity result --not even H\"older continuity-- is known to the best of our knowledge. It is the goal of this article to overcome this restrictive requirement and show regularity under minimal assumptions on the modulating coefficients. Inspired by the seminal work of Fabes, Kenig and Serapioni~\cite{FKS} we aim for a Muckenhoupt-type condition. They considered solutions of the general linear equations
\begin{align}
  \label{eq:FKS}
	-\divergence(A(x)Du) = 0,
\end{align}
with ellipticity $\lambda w(x)|\xi|^2 \le \langle A(x)\xi, \xi \rangle \le \Lambda w(x)|\xi|^2$. They proved that if the weight~$w$ is in the Muckenhoupt class~$A_2$, then the solutions are locally H\"{o}lder continuous. This result was generalized to the weighted $p$-Laplacian equation
\begin{align*}
	-\divergence\big(\langle A(x)\nabla u, \nabla u\rangle^{\frac{p-2}{2}} A(x)\nabla u\big) = 0
\end{align*}
under the condition $w^{\frac p2} \in A_p$ in~\cite{CMN}. We also recall that Modica \cite{Modica} was the first to establish H\"older continuity for quasi-minimizers of degenerate functionals controlled by Muckenhoupt weights. Results on higher regularity exist, but they require additional regularity assumptions on~$A(x)$, which we try to avoid in this article, see~\cite{BalciDieningGiovaPassarelli2022} and the references therein.

The results of \cite{FKS,CMN} suggest that a similar Muckenhoupt-type condition should be enough to guarantee H\"older continuity for the solutions of our double phase problem. This is exactly the purpose of this article. 
We will introduce a suitable Muckenhoupt condition, denoted by $\mathcal A$ (see \eqref{eq:classA}), and prove the following H\"older continuity result of  the weak solutions of \eqref{eq:pde}.
\begin{theorem}
  [H\"older regularity] \label{thm:local_Holder}
  Let $\phi(x,t)=\frac{1}{p}t^p+\frac{1}{q}a(x)t^q$, where $1<p<q$, $a:\RR^n\to [0,\infty)$, and $\phi\in\mathcal A$. If $u \in W^{1,\phi}(\Omega)$ is a weak solution of \eqref{eq:pde}, then $u$ is locally H\"older continuous. Moreover, there exists $\beta\in(0,1)$ depending on $n$, $p$, $q$ and $[\phi]_{\mathcal A}$ such that for every $4B\Subset \Omega$ with $\int_{4B}\phi(x,|\nabla u|)\, dx \leq 1$, we have that
$$
\|u\|_{L^\infty(B)} + r^\beta  [u]_{C^{0,\beta}(B)}  \le c r (M_{4B}\phi)^{-1}\left(\dashint_{4B} \phi\Big(x,\frac{|u-\mean{u}_{4B}|}{r}\Big)\, dx \right) + c M_{4B}u
$$
for some $c>0$ depending on $n$, $p$, $q$, and $[\phi]_{\mathcal A}$, where $r>0$ denotes the radius of $B$, $M_{4B}\phi(t)\coloneq \dashint_{4B}\phi(x,t)\,dx$, $M_{4B}u\coloneq \dashint_{4B}\abs{u(x)}\,dx$, $\mean{u}_{4B} \coloneq \dashint_{4B} u(x) \, dx$, and
$$
  [u]_{C^{0,\beta}(B)} := \sup_{x,y\in B,\, x\neq y} \frac{|u(x)-u(y)|}{|x-y|^\beta}.
$$
\end{theorem}
As mentioned above the famous paper of~\cite{FKS} proves H\"older regularity of solution of the weighted linear equation~\eqref{eq:FKS}. They realized that the Moser iteration technique as well as the De~Giorgi truncation technique can be applied, once a certain set of \Poincare{} and Sobolev-\Poincare{} inequalities are established. So the main work in their article consists in deriving those estimates for Lebesgue and Sobolev spaces with Muckenhoupt weights.

Our situation is quite similar. The hard step is to derive certain \Poincare{} and Sobolev-\Poincare{} inequalities that are strong enough to apply the De~Giorgi technique. However, for the double phase problem~\eqref{eq:pde} we need to work in the context of generalized Orlicz spaces. Different from the case of weighted Lebesgue spaces, our spaces are not homogeneous. This adds a significant amount of difficulties to our setting.  Let us already mention that the main step in the derivation of \Poincare{} and Sobolev-\Poincare{} inequalities is a generalized version of Jensen's inequality, see Theorem~\ref{thm:Jensen}.

As in the paper~\cite{FKS} we want to work under minimal regularity assumptions on the modulating coefficient. This requires to introduce Muckenhoupt type condition~$\phi \in \mathcal{A}$. This condition provides the optimal setup for the De~Giorgi technique under minimal assumptions. We explain below more on this new Muckenhoupt condition adapted to the double phase problem. The use of this condition allows us to treat even discontinuous modulating coefficients (see Examples~\ref{exa:Aq} and \ref{exa:min-max}).

Let us mention that the lack of homogeneity complicates the application of the De~Giorgi method at several steps. The \Poincare{} and Sobolev~\Poincare{} inequalities only hold for functions with bounded norm. However, the scaling argument, which requires homogeneity, is a fundamental tool in the De~Giorgi technique when passing from~$L^\infty$ estimates to H\"older estimates. This difficulty significantly complicates the proof of the density improvement lemma, Lemma~\ref{lem:density0}.

Furthermore, we emphasize the flexibility of our approach. As outlined above, our technique hinges fundamentally on the generalized Jensen's inequality, see Theorem~\ref{thm:Jensen}. Consequently, our arguments are in general not limited to the double phase structure and can be applied to other models satisfying this version of generalized Jensen's inequality.  For an excellent overview of regularity results for minimizers in generalized Orlicz-Sobolev spaces for various models and various conditions stronger than Muckenhoupt, see~\cite{HastoOk2022,hasto2025meanoscillationconditionsnonlinear}. Note that a higher integrability result in the weighted setting under generalized Orlicz growth conditions was recently established in \cite{HietanenLee}.

The question of finding a natural Muckenhoupt condition for the double phase model is quite delicate and has to be considered in the larger framework of generalized Orlicz spaces. We keep the notation in this introduction brief, but more details and exact definitions are provided in  Section~\ref{sec:double-phase-model}. The standard weighted Lebesgue spaces correspond to the generalized Young function $\phi(x,t) = w(x)t^p$. Then the standard Muckenhoupt condition~$w \in A_p$ is rephrased in terms of indicator functions of cubes, i.e. we say $\phi \in \mathcal{A}$ if
\begin{align}
  \label{eq:classApre}
  [\phi]_{\mathcal{A}}:= \sup_Q \frac{\norm{\indicator_Q}_\phi \norm{\indicator_Q}_{\phi^*}}{\abs{Q}} < \infty,
\end{align}
where $\norm{\cdot}_\phi$ is the norm in the generalized Orlicz space~$L^\phi$ and $\phi^*$ is the conjugate of~$\phi$. This definition extends naturally to all generalized Orlicz spaces. However, it is not obvious if it ensures similar good properties as in the case of weighted Lebesgue spaces.
In fact, if $w \in A_p$ and $1<p <\infty$, then the maximal operator~$M$ (see \eqref{eq:Mf}) is bounded on the weighted space $L^p_w$ and its dual $L^{p'}_\sigma$ with $p'=\frac{p}{p-1}$ and $\sigma= \omega^{-\frac{p'}{p}}$. In other words~\eqref{eq:classApre} implies the boundedness of $M$ on $L^\phi$ and $L^{\phi^*}$ for $\phi(x,t) = w(x)t^p$.

Our Muckenhoupt condition~\eqref{eq:classApre} also appears in the context of variable exponent spaces~$L^{p(\cdot)}$, where $\phi(x,t) = \tfrac{1}{p(x)} t^{p(x)}$. These spaces have been studied intensively in the last twenty years. Most papers assume the so-called log-H\"older continuity of the exponent together with certain decay conditions at infinity, which implies the boundedness of~$M$, see \cite{Die2,UribeFiorenzaNeugebauer2003,Nekvinda2004}. We refer in particular to the books~\cite{DHHR,CUF}. It turns out that the Muckenhoupt condition~$\phi \in \mathcal{A}$ is not sufficient for the boundedness of~$M$ on~$L^{p(\cdot)}$, see \cite[Theorem~5.3.4]{DHHR}. Boundedness of~$M$ can only be guaranteed under an additional decay condition at infinity, see \cite{Lerner2010questions,AdamadzeDieningKopaliani2026} and \cite[Theorem~4.52]{CUF}. We will see in Theorem~\ref{thm:M_bounded} that surprisingly  the Muckenhoupt condition $\phi \in \mathcal{A}$ is equivalent to the boundedness of~$M$ in the case of the double phase model without any additional condition. Note that the boundedness of~$M$ on weighted generalized Orlicz spaces with classical $A_p$ weights was established in \cite{VerttiHietanen}.

Note that most of the papers on variable exponents and their related PDE assume the stronger condition  log-H\"older continuity on the variable exponent, see~\cite[Definition~4.1.1]{DHHR}. Passing from the $C^{0,\alpha}$-assumption on the modulating coefficient for the double phase problem to a Muckenhoupt condition is similar to passing from log-H\"older continuity to a Muckenhoupt condition for variable exponents. However, up to now there exist no regularity results for PDE with variable exponents assuming a Muckenhoupt type condition.

The techniques from variable exponent spaces based on log-H\"older continuity have been generalized to the case of generalized Orlicz spaces. Most prominent is the book of H\"asto and Harjulehto, see \cite{HastoHarjulehto2019}. The assumptions on~$\phi(x,t)$ are given by three axioms. Axiom (A1) and (A2) generalize the local log-H\"older continuity and the log decay from the case of variable exponents. The third axiom (A0) states basically that $\phi(x,1) \eqsim 1$, which excludes the weighted case. A generalization of this theory that includes all weighted spaces is still out of reach. Since our example for the double phase model in Section~\ref{sec:examples} also includes unbounded modulating coefficients, we are in the setting without~(A0). This is an important step towards the general situation.

In this paper we also include some important results on the boundedness of the maximal operator on generalized Orlicz spaces. As we mentioned above the maximal operator is bounded on weighted Lebesgue spaces $L^p_\omega$ if and only if it is bounded on its dual space $L^{p'}_\sigma$. The same property surprisingly holds for variable exponent spaces~\cite{Die1} and \cite[Chapter~5]{DHHR}. The property that $M$ is bounded on~$L^\phi$ and its dual $L^{\phi^*}$ (duality property of~$M$) is of fundamental importance, since it is equivalent to the boundedness of \Calderon-Zygmund operators and the Riesz transforms, see \cite[Theorem~1.2]{Zoe}. The question of the duality property of~$M$ has been raised and studied by Lerner. In~\cite{Ler17} he considers the case of weighted variable Lebesgue spaces. He proved the duality property for these spaces but he unfortunately requires an additional~$A_\infty$ assumption. The open question of the duality property of~$M$ is stated in~\cite{Zoe} in greater generality. In this article, we show that the double phase function space has indeed the duality property of~$M$, see Theorem~\ref{thm:M_bounded}. This is a consequence of the invariance of the Muckenhoupt condition~$\phi \in \mathcal{A}$ under duality and the characterization of the boundedness of~$M$ by $\phi \in \mathcal{A}$.

In the case of generalized Orlicz spaces there are two types of boundedness of~$M$. The classical one is in terms of norms, i.e. $\norm{Mf}_\phi \lesssim \norm{f}_\phi$. The second is the so-called modular boundedness, which is stronger and more useful for variational problems, i.e.
\begin{align}
  \label{eq:modular-boundeness}
	\int_{\RRn}\phi(x,\abs{Mf(x)})dx \lesssim \int_{\RRn}\phi(x,\abs{f(x)})dx.
\end{align}
Although this estimate holds for weighted~$L^p$ spaces as well as for (unweighted) Orlicz spaces, in general, it is too strong to be valid. One has to replace it with restricted modular boundedness
\begin{align}
  \label{eq:modular-boundeness-restricted}
	\int_{\RRn}\phi(x,\abs{Mf(x)})dx \lesssim \int_{\RRn}\phi(x,\abs{f(x)})dx \qquad \text{for all $f$ with $\norm{f}_\phi  \leq 1$}.
\end{align}
Even the restricted inequality \eqref{eq:modular-boundeness-restricted} is too strong for certain generalized Orlicz spaces. For instance, in variable exponent spaces, \eqref{eq:modular-boundeness-restricted} holds only if the exponent is constant; see~\cite{Lerner2005}. This is proved by Lerner for~\eqref{eq:modular-boundeness} but the proof also applies to~\eqref{eq:modular-boundeness-restricted}.  Thus, it is rather surprising that restricted modular boundedness of~$M$ holds for the double phase model, see Theorem~\ref{thm:M_bounded}.

Overall, we extend the work of \cite{FKS,CMN} to the setting of the double phase model under the minimal Muckenhoupt assumption on the modulating coefficient. Moreover, we characterize the boundedness of the maximal operator completely by this Muckenhoupt condition. In Section~\ref{sec:double-phase-model} we combine all results related to the function spaces. In Section~\ref{sec:regularity-results} we prove the H\"older continuity of the solutions of~\eqref{eq:pde}.
	
\section{The double phase model and function spaces}
\label{sec:double-phase-model}

In this section we introduce the double phase model and investigate the corresponding function spaces. In particular, we will introduce the Muckenhoupt condition that characterizes the boundedness of the maximal operator. We also investigate left-openness results, the boundedness of the maximal operator on the dual space and Sobolev--Poincar\'e-type inequalities.

\subsection{The model}

We start with a bit of notation. Let $\Omega$ be any open subset of $\RRn$. By $L^0(\Omega)$ we denote the space of measurable functions on~$\Omega$. By $L^p(\Omega)$ we denote the usual Lebesgue space and by $L^1_{\loc}(\Omega)$ the space of locally integrable functions ($L^1$ on compact subsets). 


We can now describe the double phase model as introduced in \cite{Zhik1}.  For $1<p<q<\infty$ and a weight $a \in L^1_{\loc}(\RRn)$ with $a \geq 0$, 
set
\begin{align}
  \phi(x,t)=\tfrac 1p t^p+ \tfrac 1q a(x)  t^q.
\end{align}
We remark that, by a slight abuse of notation, we use the symbol $\phi$ to denote both the double phase function and general $N$-functions throughout the paper. However, whenever the precise form is required, it will be explicitly stated.
We are later interested in the regularity of weak solutions of the Euler--Lagrange equation corresponding to the energy
\begin{align*}
  \mathcal{J}(v) = \int_\Omega \phi(x,\abs{\nabla v})\,dx.
\end{align*}
Although the PDE later is defined on a domain~$\Omega \subset \RRn$ we assume for the sake of simplicity that the coefficient $a$ is defined on all of~$\RRn$.

If $a=0$, then the model just describes the $p$-Laplacian. For $a > 0$ the model belongs to the class of $p$-$q$ growth problems introduced and investigated by~\cite{Mar}. These problems have a lower $p$-growth and an upper $q$-growth. For general coefficients the solutions to the corresponding PDE can be quite irregular and the Lavrentiev gap phenomenon can occur \cite{Zhik3,ELM,BDS}. To avoid this problem and to show some regularity of the solution, some condition on the modulating coefficient has to be assumed. Usually, a certain type of H\"older continuity for the coefficient~$a$ is assumed. We will avoid such a strong regularity for the coefficient~$a$ and replace it later by a Muckenhoupt-type condition, which we will introduce in Section~\ref{sec:muck-cond}.

Our function~$\phi$ is a generalized N-function (a sub-class of generalized Young functions). The corresponding spaces have been investigated for example in~\cite[Section~2]{DHHR} and \cite{HastoHarjulehto2019}. We briefly repeat some basic properties of this theory. A function~$\phi\,:\, \RRn \times [0,\infty) \to [0,\infty)$ is a generalized N-function if
\begin{enumerate}
\item $x \mapsto \phi(x,t)$ is measurable for all $t \geq 0$.
\item $t \mapsto \phi(x,t)$ is an N-function for a.e. $x \in \RRn$, i.e. it is continuous, convex, and satisfies $ \lim_{t\to 0}\frac{\phi(x,t)}{t}=0$ and $\lim_{t\rightarrow \infty}\frac{\phi(x,t)}{t}=\infty$.  It has a right derivative $\phi'(x,t)$ with $\phi(x,t) = \int_0^t \phi'(x,s)\,ds$.
\end{enumerate}
By $(\phi')^{-1}(x,\cdot)$ we denote the right-continuous inverse of $t \mapsto \phi'(x,t)$. We define the function $\phi^*(x,t) \coloneqq \int_0^t (\phi')^{-1}(x,s)\,ds$, which is itself a generalized $N$-function satisfying $(\phi^*)' = (\phi')^{-1}$. This function is known as the conjugate $N$-function and can equivalently be characterized by
\begin{align}
	\label{eq:conjugate}
	\phi^*(x,s) = \sup_{t \geq 0} \big(st - \phi(x,t)\big) \qquad \text{for all $x \in \RRn$ and $s \geq 0$}.
\end{align}
Moreover, $(\phi^*)^* = \phi$. The estimate $st \leq \phi(x,t) + \phi^*(x,s)$ is known as Young's inequality. We use the notation $f \lesssim g$ and $f \eqsim g$ if $f \le c g$ and $c_1 f \le g \le c_2 f$ hold for some positive constants $c, c_1, c_2$, respectively.

It is difficult to write the conjugate function $\phi^*$ for our double phase model explicitly. However, we have the following characterization.
\begin{lemma}
  \label{lem:phi*}
  Let $\phi(x,t)=\tfrac 1p t^p+ \tfrac 1q a(x)t^q$ with $1<p<q < \infty$ and $a\,:\, \RRn \to [0,\infty)$ be measurable.  There hold
  \begin{align*}
    (\phi^*)'(x,t) &\eqsim \min\bigset{t^{p'-1},\,\, a(x)^{1-q'}t^{q'-1}},
    \\
    \phi^*(x,t) &\eqsim \min\bigset{t^{p'},\,\, a(x)^{1-q'}t^{q'}}.
  \end{align*}
  The implicit constants only depend on $p$ and $q$.
\end{lemma}
\begin{proof}
  We start with
  \begin{align*}
    \phi'(x,t)=t^{p-1}+a(x)t^{q-1} \eqsim  \max\bigset{t^{p-1},\,\,a(x)t^{q-1}}.
  \end{align*}
  The formula $(\phi^*)' = (\phi')^{-1}$ then implies
  \begin{align*}
    (\phi^*)'(x,t)\eqsim \min\bigset{t^{p'-1},\,\, a(x)^{1-q'}t^{q'-1}}
  \end{align*}
  with the convention $0^{1-q'} = \infty$. This proves the first part of the claim. The formula $\phi^*(x,t) \eqsim (\phi^*)'(x,t)\,t$ proves the second part.
\end{proof}            
A direct calculation shows that
\begin{alignat}{2}
	\label{eq:index-sim}
	p &\leq \;\;\frac{\phi'(x,t)\,t}{\phi(x,t)} &&\leq q,
	\\  \label{eq:index-unif}
	p &\leq \frac{\phi''(x,t)\,t}{\phi'(x,t)} + 1 &&\leq q.
\end{alignat}
The first inequality~\eqref{eq:index-sim} says that $p$ and~$q$ are lower and upper Simonenko indices (note that on the set $\{a=0\}$, both indices are equal to $p$). The second inequality~\eqref{eq:index-unif} basically says that $\phi(x,\cdot)$ is uniformly convex with indices~$p$ and~$q$ (similarly, on the set $\{a=0\}$, both indices are equal to $p$), compare\footnote{Translated to~\cite{DFTW} we have~$c_6=p-1$ and $c_7=q-1$.} \cite[Appendix~B]{DFTW}.
The estimates~\eqref{eq:index-sim} and~\eqref{eq:index-unif} are equivalent to the following ones: For all $s,t \geq 0$
\begin{alignat}{2}
  \label{eq:lower-upper1}
  \min \set{s^p,s^q} \phi(x,t) \leq
  \phi(x,st) &\leq \max \set{s^p,s^q} \phi(x,t),
  \\
  \label{eq:lower-upper2}
  \min \set{s^{p-1},s^{q-1}} \phi'(x,t) \leq
  \phi'(x,st) &\leq \max \set{s^{p-1},s^{q-1}} \phi'(x,t).
\end{alignat}
In particular, $\phi$ satisfies the $\Delta_2$-condition $\phi(x,2t)\leq 2^q \phi(x,t)$.
The conjugate~$\phi^*$ is also uniformly convex with indices $q'= \frac{q}{q-1}$ and $p'=\frac{p}{p-1}$, see \cite[Lemma~29]{DFTW}. Thus, the formulas~\eqref{eq:index-sim}, \eqref{eq:index-unif}, \eqref{eq:lower-upper1} and \eqref{eq:lower-upper2} hold with $\phi$, $p$, $q$ replaced by $\phi^*$, $q'$ and $p'$, respectively.

\subsection{Generalized Orlicz function spaces}
The generalized N-function~$\phi$ defines a generalized Orlicz space, as detailed in~\cite[Section~2]{DHHR} and \cite{HastoHarjulehto2019}. 
For a function $f \in L^0(\Omega)$ we define the modular
\begin{align}
  \label{eq:modular}
  \rho_\phi(f) &:= \int_\Omega  \phi(x, \abs{f(x)})\,dx
\end{align}
and the Luxemburg norm
\begin{align*}
  \norm{f}_{L^{\phi}(\Omega)}\coloneqq\inf\{\lambda>0\,:\rho_\phi(|f|/\lambda)\leq1\}.
\end{align*}
For convenience, we write $\norm{f}_{\phi} := \norm{f}_{L^{\phi}(\RRn)}$. Then the generalized Orlicz space $L^{\phi}(\Omega)$ (also known as the Musielak-Orlicz space) is defined by
\begin{align*}
  L^{\phi}(\Omega)\coloneqq\bigset{ f \in L^0(\Omega)\,:\, \norm{f}_{L^{\phi}(\Omega)}<\infty}.
\end{align*}
The generalized Orlicz-Sobolev space $W^{1,\phi}(\Omega)$ is the set of all functions $u \in L^\phi(\Omega)$ for which the weak derivatives $D_ju$ belong to $L^\phi(\Omega)$ for each $j=1,\dots,n$, equipped with the norm $\|u\|_{W^{1,\phi}(\Omega)} = \|u\|_{L^\phi(\Omega)} + \|Du\|_{L^\phi(\Omega)}$. We denote by $W_0^{1,\phi}(\Omega)$ the space defined by the intersection $W^{1,1}_0(\Omega)\cap W^{1,\phi}(\Omega)$, equipped with the norm $\|\cdot\|_{W^{1,\phi}(\Omega)}$. Alternatively, $W_0^{1,\phi}(\Omega)$ can be characterized as the set of functions $\{ u \in W^{1,\phi}(\Omega) : \tilde{u} \in W^{1,\phi}(\mathbb{R}^n) \}$, where $\tilde{u}$ denotes the extension of $u$ by zero to $\mathbb{R}^n \setminus \Omega$.
It is well-known that $L^\phi(\Omega)$ and $W^{1,\phi}(\Omega)$ are uniformly convex Banach spaces under suitable assumptions on $\phi$ (see, e.g., \cite[Remark~2.4.15]{DHHR}; see also \cite[Theorem 6.1.4]{HastoHarjulehto2019}). We note that these assumptions hold for the double phase model considered in this paper.
For all $f\in L^{\phi}(\Omega)$ and $g\in L^{\phi^{\ast}}(\Omega)$ we have the H\"{o}lder inequality
\begin{align}
  \label{eq:hoelder}
  \int_{\Omega}\abs{f} \,\abs{g} \,dx\leq 2 \|f\|_{L^{\phi}(\Omega)}\|g\|_{L^{\phi^{\ast}}(\Omega)},
\end{align}
see \cite[Lemma~2.6.5]{DHHR}.

Another useful equality is the \emph{unit ball property}, see \cite[Lemma~2.1.14]{DHHR}, i.e. for all $f \in L^\phi(\Omega)$ we have
\begin{equation}
	\label{eq:unit-ball}
	\int_{\Omega} \phi\bigl(x, \abs{f(x)} / \norm{f}_{L^\phi(\Omega)}\bigr)\,dx = 1.
\end{equation}
 Also, we have the following useful characterization of the norm, see \cite[Corollary~2.7.5]{DHHR}:
	\begin{lemma}[Norm conjugate formula]
		\label{lem:norm-conjugate-formula}
		Let $\phi$ be a generalized $N$-function. Then for all $f \in L^0(\Omega)$ there holds
		\begin{align*}
			\norm{f}_{L^\phi(\Omega)} \leq \sup_{g \in L^{\phi^*}(\Omega)\,:\, \norm{g}_{L^{\phi^*}(\Omega)} \leq 1} \int_\Omega \abs{f g}\,dx \leq 2 \norm{f}_{L^\phi(\Omega)}.
		\end{align*}
	\end{lemma}
\subsection{Muckenhoupt condition}
\label{sec:muck-cond}
As mentioned in the introduction, we are interested in the boundedness of the Hardy--Littlewood maximal operator, which is defined for $f\in L^{1}_{\loc}(\RRn)$ by
\begin{align}
  \label{eq:Mf}
  Mf(x)=\sup_{x\in Q}\dashint_Q|f(y)|\,dy = \sup_{x\in Q} \frac{1}{\abs{Q}} \int_Q|f(y)|\,dy,
\end{align}
where $\abs{Q}$ is the Lebesgue measure of~$Q$. The supremum is taken over all cubes (with sides parallel to the coordinate axes) containing the point $x$. It would also be possible to use balls or centered balls.

We say that $\phi$ satisfies the Muckenhoupt condition, i.e. $\phi \in \mathcal{A}$, if
\begin{align}
  \label{eq:classA}
  [\phi]_{\mathcal{A}}:= \sup_Q \frac{\norm{\indicator_Q}_\phi \norm{\indicator_Q}_{\phi^*}}{\abs{Q}} < \infty,
\end{align}
where the supremum is taken over all cubes $Q$ in $\RRn$ and $\indicator_Q$ is the indicator function of $Q$. (One could also use balls.) This condition is denoted by $\mathcal{A}_{\loc}$ in \cite{DHHR}. Moreover, $\phi \in \mathcal{A}$ is equivalent to $\phi^* \in \mathcal{A}$ and $[\phi]_{\mathcal{A}} = [\phi^*]_{\mathcal{A}}$.
\begin{remark}
\label{rem:muckenhoupt-classical}
The condition $\phi \in \mathcal{A}$ is a natural generalization of the classical Muckenhoupt condition for a weight $w$. In particular, for $1<p < \infty$, a weight $w$ belongs to the class $A_p$ by definition if
\begin{align}
\label{eq:muckenhoupt-classical}
[w]_{A_{p}}\coloneq\sup_Q\Bigg( \bigg(\fint_Qw(x)\,dx\bigg)\bigg(\fint_Qw(x)^{-\frac{p'}{p}}\,dx\bigg)^{\frac{p}{p'}}\Bigg)<\infty.
\end{align}
Then one can show that $w \in A_p$ if and only if $\phi \in \mathcal{A}$ with $\phi(x,t) = w(x) t^p$.
\end{remark}
In Section~\ref{sec:examples} we will provide specific examples for the modulating coefficient~$a$ such that $\phi \in \mathcal{A}$ for our double phase model $\phi(x,t) = \tfrac 1p t^p + \frac 1q a(x)t^q$.
Note that $\phi \in \mathcal{A}$ is equivalent to the uniform boundedness of the operators $f \mapsto \indicator_Q \dashint_Q \abs{f}\,dx$ on $L^\phi(\RRn)$. This follows exactly as in \cite[Theorem~4.5.7]{DHHR}. 
\begin{remark}
  \label{rem:A-additive}
  In the general setting of generalized Young functions, we obtain the following interesting result that the class $\mathcal{A}$ is closed under addition. 
  First, we observe that $\norm{f}_{\phi_1+\phi_2} \leq \norm{f}_{\phi_1} + \norm{f}_{\phi_2}$ and $\norm{f}_{(\phi_1+\phi_2)^*} \leq \min \set{ \norm{f}_{\phi_1^*}, \norm{f}_{\phi_2^*}}$. This implies
  \begin{align*}
    [\phi_1+\phi_2]_{\mathcal{A}}& = \sup_Q \frac{\norm{\indicator_Q}_{\phi_1+\phi_2}\, \norm{\indicator_Q}_{(\phi_1+\phi_2)^*}}{|Q|}
    \\
    &\le
    \sup_Q \frac{(\norm{\indicator_Q}_{\phi_1} + \norm{\indicator_Q}_{\phi_2})\,\min(\norm{\indicator_Q}_{\phi_1^*}, \norm{\indicator_Q}_{\phi_2^*})
    }{|Q|}
    \\
    &\le
    \sup_Q \frac{\norm{\indicator_Q}_{\phi_1}\norm{\indicator_Q}_{\phi_1^*}
    }{|Q|} + \sup_Q 
    \frac{\norm{\indicator_Q}_{\phi_2}\norm{\indicator_Q}_{\phi_2^*}}{|Q|}
    \\
    &\le [\phi_1]_{\mathcal{A}} + [\phi_2]_{\mathcal{A}}.
  \end{align*}
  Moreover,
  \begin{align*}
    [(\phi_1^*+\phi_2^*)^*]_{\mathcal{A}} =
    [\phi_1^*+\phi_2^*]_{\mathcal{A}} \leq 
    [\phi_1^*]_{\mathcal{A}} +
    [\phi_2^*]_{\mathcal{A}} =
    [\phi_1]_{\mathcal{A}} +
    [\phi_2]_{\mathcal{A}}.
  \end{align*}
\end{remark}  

Let us introduce a bit of useful notation. For a cube~$Q$ (or ball) and $f \in L^1_{\loc}(\RRn)$ we define
\begin{align*}
  M_Q f &:= \dashint_Q \abs{f(x)}\,dx,
  \\
  (M_Q\phi)(t) &\coloneq\dashint_Q\phi(x,t)\,dx
\end{align*}
Note that $M_Q \phi$ is an N-function; it is the averaged version of~$\phi$. It follows from~\eqref{eq:lower-upper1} that $M_Q \phi$ has the same lower and upper Simonenko indices~$p$ and~$q$, respectively.
	Analogously, we define the N-function $M_Q\phi^*$ as the averaged version of~$\phi^*$ which has Simonenko indices $q'$ and $p'$.
	As usual we denote by $(M_Q\phi^*)^*$ the conjugate N-function of $M_Q\phi^*$; it has Simonenko indices~$p$ and~$q$. 
	It has been shown in \cite[Lemma~5.2.8]{DHHR} that
	\begin{align*}
		(M_Q \phi^*)^*(M_Q f) \leq M_Q (\phi(f)) = \dashint_Q \phi(x,\abs{f(x)})\,dx
	\end{align*}
	and
	\begin{align}
		\label{eq:MQJensen}
		(M_Q \phi^*)^*(t) = \inf_{f \in L^0(Q) : M_Q f =t}  M_Q (\phi(f)) \leq (M_Q\phi)(t).
	\end{align}
	Note that $(M_Q \phi^*)^*$ is also referred to as the infimal convolution, see \cite{Kopaliani}.
Recall that a generalized Young function $\phi$ belongs to the class $\mathcal{A}_{\text{glob}}$, denoted $\phi \in \mathcal{A}_{\text{glob}}$, if the averaging operators $T_{\mathcal{Q}}$ are bounded on $L^{\phi}(\mathbb{R}^n)$ uniformly over all disjoint families $\mathcal{Q}$ of cubes in $\mathbb{R}^n$. The operator $T_{\mathcal{Q}}$ is defined for a function $f$ as:
\begin{align*}
  T_{\mathcal{Q}}f = \sum_{Q\in \mathcal{Q}} \indicator_Q M_Q f = \sum_{Q\in \mathcal{Q}} \indicator_Q \dashint_Q|f(y)| \, dy.
\end{align*}

\subsection{Generalized Jensen's inequality for the double phase}

The main result of this section is the following theorem.
\begin{theorem}[Generalized Jensen's inequality]
  \label{thm:Jensen}
  Let $\phi(x,t)=\tfrac 1p t^p+ \tfrac 1q a(x)t^q$ with $1<p<q < \infty$ and $a\,:\, \RRn \to [0,\infty)$ be measurable. Then the following assertions are equivalent:
  \begin{enumerate}
  \item \label{itm:Muck_a} $\phi \in \mathcal{A}$.
  \item \label{itm:Jensen_b} There exists $C>0$ such that for every $f\in L^{\phi}(\RRn)$ with $\norm{f}_\phi \leq 1$ and for every cube $Q$ in $\RRn$,  we have
    \begin{align}
      \label{eq:generalized-jensen-phi}
      \int_Q\phi\left(x,\frac{1}{|Q|}\int_Q|f(y)|\,dy\right)\,dx\leq C\int_Q\phi(x,|f(x)|)\,dx.
    \end{align}
  \item \label{itm:Jensen_c} There exists $C>0$ such that for every $f\in L^{\phi^*}(\RRn)$ with $\norm{f}_{\phi^*} \leq 1$ and for every cube $Q$ in $\RRn$,  we have
    \begin{align}
      \label{eq:generalized-jensen-phi*}
      \int_Q\phi^*\left(x,\frac{1}{|Q|}\int_Q|f(y)|\, dy\right)\, dx\leq C\int_Q\phi^*(x,|f(x)|)\, dx.
    \end{align}
      \end{enumerate}
 In particular, if one of (a)--(c) is valid, then so are the others, with constants 
$[\phi]_{\mathcal{A}}$  or $C$  depending only on
$p$, $q$, and the constant in the assumed statement.
\end{theorem}

\begin{remark}
  Note that pointwise estimates, such as those in \cite[Theorem 4.3.2]{HastoHarjulehto2019} (see also \cite[Theorem 4.2.4]{DHHR}), do not hold without additional assumptions on the $\phi$-function. The validity of such estimates relies on conditions like (A0), (A1), and (A2); see \cite{HastoHarjulehto2019}. For instance, in the double phase model (a special case of generalized Orlicz spaces), when the modulating coefficient $a$ is bounded and $a \in C^{0,\alpha}(\mathbb{R}^n)$ with $\alpha\in(0, 1]$, we obtain the pointwise key estimate \cite[Theorem 4.3.2]{HastoHarjulehto2019} provided that the exponents satisfy $\frac{q}{p} \leq 1 + \frac{\alpha}{n}$.
\end{remark}
For the proof of Theorem~\ref{thm:Jensen} we need the following auxiliary result.

\begin{lemma}\label{lem:intermedinequality}
  Let $\phi \in \mathcal{A}$ where $\phi(x,t)=\tfrac 1p t^p+ \tfrac 1q a(x)t^q$. Then for all $t$ with $0 \leq t \leq \frac{1}{\norm{\indicator_Q}_{\phi^*}}$ we have
  \begin{align}
    \dashint_Q\phi'\bigg(x,\dashint_Q(\phi^*)'(y,t)\,dy\bigg)dx \lesssim t.
  \end{align}
  The implicit constant only depends on $p$, $q$ and $[\phi]_{\mathcal{A}}$.
\end{lemma}
\begin{proof}
  We have
  \begin{align*}
    \dashint_Q(\phi^*)'(y,t)\,dy\eqsim\dashint_Q\min\bigset{t^{p'-1},\,\, a(y)^{1-q'}t^{q'-1}}\,dy.
  \end{align*}
  Hence, with $(p'-1)(p-1)=1$ we obtain
  \begin{align*}
    \lefteqn{\dashint_Q\phi'\left(x, \dashint_Q(\phi')^{-1}(y,t)dy\right)dx} \qquad &
    \\
    &\eqsim\dashint_Q\left(\dashint_Q\min\bigset{t^{p'-1},\,\, a(y)^{1-q'}t^{q'-1}}\,dy
    \right)^{p-1}\,dx
    \\
    &\quad +\dashint_Qa(x)\left(\dashint_Q\min\bigset{t^{p'-1},\,\, a(y)^{1-q'}t^{q'-1}}\,dy \right)^{q-1}\,dx
    \\
    &\lesssim t  +\dashint_Q a(x)\left(\dashint_Q\min\bigset{t^{p'-1},\,\, a(y)^{1-q'}t^{q'-1}}dy\right)^{q-1}\,dx
    \\
    &\eqsim t\, \Bigg(1 +\dashint_Q a(x)\left(\dashint_Q\min\{t^{p'-q'},\,\, a(y)^{1-q'}\}dy\right)^{q-1}\,dx\Bigg).
  \end{align*}
  It remains to prove that
  \begin{align*}
    \dashint_Q a(x)\left(\dashint_Q\min\{t^{p'-q'},\,\, a(y)^{1-q'}\}\,dy\right)^{q-1}\,dx &\lesssim 1.
  \end{align*}
  We calculate
  \begin{align*}
    \lefteqn{\dashint_Q\min\bigset{t^{p'-q'},\,\, a(y)^{1-q'}}\,dy} \qquad &
    \\
    &\leq \abs{Q}^{-1}\int_Q\min\bigset{(1/\norm{\indicator_Q}_{\phi^*})^{p'-q'},\,\, a(y)^{1-q'}}dy
    \\
    &=\abs{Q}^{-1}\norm{\indicator_Q}_{\phi^*}^{q'}\int_Q\min\{(1/\|\indicator_Q\|_{L^{\phi^{\ast}}})^{p'},\,\, a(y)^{1-q'}(1/\|\indicator_Q\|_{L^{\phi^{\ast}}})^{q'}\}dy
    \\
    &\eqsim \abs{Q}^{-1}\norm{\indicator_Q}_{L^{\phi^{\ast}}}^{q'}\int_Q\phi^*\bigg(y, \frac{1}{\norm{\indicator_Q}_{\phi^*}}\bigg)\,dy
    \\
    &=\abs{Q}^{-1} \norm{\indicator_Q}_{\phi^*}^{q'},
  \end{align*}
  where we have used the unit ball property in the last step.
  Thus,  
  \begin{align*}
    \dashint_Q a(x)\left(\dashint_Q\min\bigset{t^{p'-q'},\,\, a(x)^{1-q'}}\,dx\right)^{q-1}\,dy
    &\lesssim \dashint_Qa(x) \abs{Q}^{1-q} \norm{\indicator_Q}_{\phi^*}^q \,dx
    .
  \end{align*}
  Since $\phi \in \mathcal{A}$, we have $ \abs{Q} \eqsim \norm{\indicator_Q}_{\phi^*} \norm{\indicator_Q}_\phi$. Using this fact, the previous estimate, $(q-1)(q'-1)=1$ and the unit ball property, it implies that 
  \begin{align*}
    \lefteqn{\dashint_Q a(x)\left(\dashint_Q\min\bigset{t^{p'-q'},\,\, a(x)^{1-q'}}\,dx\right)^{q-1}\,dy} \qquad &
    \\
    &\lesssim \int_Qa(x) \norm{\indicator_Q}_\phi^{-q} \,dx
    \lesssim \int_Q \phi\bigg(x, \frac{1}{\norm{\indicator_Q}_\phi}\bigg)\,dx = 1.
  \end{align*}
  This proves the claim.
\end{proof}

\begin{proof}[Proof of Theorem \ref{thm:Jensen}]
  The implication $\ref{itm:Jensen_b}\Rightarrow\ref{itm:Muck_a}$ follows directly from \cite[Theorem~4.5.7]{DHHR}. Although the proof in \cite{DHHR} is for variable Lebesgue spaces, it carries over to the generalized Orlicz setting since its key ingredients, the norm conjugate formula Lemma~\refeq{lem:norm-conjugate-formula} and H\"{o}lder's inequality \eqref{eq:hoelder}, remain valid for the general setting. Now we prove $\ref{itm:Muck_a}\Rightarrow\ref{itm:Jensen_b}$.  Let us define $F(t) \coloneqq\int_{Q}(\phi^{\ast})'(x,t)\,dx$. Using Lemma~\ref{lem:phi*}, we can easily check that $F(0)=0$, $\lim_{t\to 0}F(t)=0$, and $\lim_{t\to\infty}F(t)=\infty$. We can also check easily that $F(t)$ is nondecreasing and continuous.
  Let $\norm{f}_\phi \leq 1$. Then, by H\"older's inequality \eqref{eq:hoelder},
  \begin{align*}
    \int_Q|f(x)|\,dx \leq 2\norm{f\indicator_Q}_\phi\norm{\indicator_Q}_{\phi^*} \leq 2\norm{\indicator_Q}_{\phi^*}.
  \end{align*}
Since $F(t) \eqsim \frac{1}{t}\int_Q\phi^{\ast}(x,t)\,dx$ and, consequently,
  \begin{align*}
 \int_Q(\phi^{\ast})'(x,1/\norm{\indicator_Q}_{\phi^*})\,dx  =    F(1/\norm{\indicator_Q}_{\phi^*}) \eqsim \norm{\indicator_Q}_{\phi^*},
  \end{align*}
there exists $\widetilde{t}>0$   with $0 < \widetilde{t} \leq c/\norm{\indicator_Q}_{\phi^*}$ such that
  \begin{align*}
    \int_Q(\phi^{\ast})'(x,\widetilde{t})\,dx = \int_Q|f(x)|\,dx,
  \end{align*}
 where the constant $c>0$ is independent of the cube $Q$ and the function $f$.
  Since $\phi(x,(\phi^{\ast})'(x,t)) \eqsim \phi^{\ast}(x,t)$, we have
  \begin{align*}
    \int_Q\phi(x,(\phi^{\ast})'(x,\widetilde{t}))\,dx \eqsim \int_Q\phi^{\ast}(x,\widetilde{t})\,dx.
  \end{align*}
  We will prove that
  \begin{align*}
    \int_Q\phi^{\ast}(x,\widetilde{t})\,dx \leq c\int_Q\phi(x,|f(x)|)\,dx.
  \end{align*}
  Note that $\widetilde{t}$ is fixed in the range $0 < \widetilde{t} \leq c/\norm{\indicator_Q}_{\phi^*}$. Since $\phi^{\ast}(x,\widetilde{t}) \eqsim \widetilde{t}(\phi^{\ast})'(x,\widetilde{t})$, we have
  \begin{align*}
    \int_Q\phi^{\ast}(x,\widetilde{t})\,dx \eqsim \widetilde{t}\int_Q(\phi^{\ast})'(x,\widetilde{t})\,dx = \widetilde{t}\int_Q|f(x)|\,dx.
  \end{align*}
Using the facts that $(\phi^{\ast})'=(\phi')^{-1}$  and $\phi' (x,t)t\eqsim \phi(x,t)$, we have
  \begin{align*}
    \widetilde{t}\int_Q(\phi^{\ast})'(x,\widetilde{t})\,dx
    &= 2\widetilde{t}\int_Q|f(x)|\,dx - \widetilde{t}\int_Q(\phi^{\ast})'(x,\widetilde{t})\,dx \\
    &\leq 2\widetilde{t}\int_{\{y\in Q \,:\, 2\widetilde{t}|f(x)| > \widetilde{t}(\phi^{\ast})'(x,\widetilde{t})\}} |f(x)|\,dx \\
    &= 2\widetilde{t}\int_{\{x\in Q \,:\, ((\phi^{\ast})')^{-1}(x,2|f(x)|) \geq \widetilde{t}\}} |f(x)|\,dx \\
    &\lesssim \int_Q ((\phi^{\ast})')^{-1}(x,2|f(x)|)|f(x)|\,dx \\
    &\eqsim \int_Q\phi'(x,2|f(x)|	)|f(x)|\,dx \\
    &\eqsim \int_Q\phi(x,|f(x)|)\,dx.
  \end{align*}
  Using Lemma \ref{lem:intermedinequality} 
  we get
  \begin{align*}
    \int_Q\phi
    &\left(x,\frac{1}{|Q|}\int_Q|f(y)|\,dy\right)dx \\
    &=\int_Q\phi\left(x,\frac{1}{|Q|}\int_Q(\phi^{\ast})'(y,\widetilde{t})\,dy\right)dx \\
    &\eqsim \int_Q\phi'\left(x,\frac{1}{|Q|}\int_Q(\phi^{\ast})'(y,\widetilde{t})\,dy\right)dx \cdot \frac{1}{|Q|}\int_Q(\phi^{\ast})'(y,\widetilde{t})\,dy \\
    &= \frac{1}{\widetilde{t}|Q|}\int_Q\phi'\left(x,\frac{1}{|Q|}\int_Q (\phi^{\ast})'(y,\widetilde{t})\,dy\right)dx \cdot \widetilde{t}\int_Q(\phi^{\ast})'(y,\widetilde{t})\,dy \\
    &\lesssim \frac{1}{\widetilde{t}|Q|}\int_Q\phi'\left(x,\frac{1}{|Q|}\int_Q(\phi^{\ast})'(y,\widetilde{t})\,dy\right)dx \cdot \int_Q\phi(y,|f(y)|)\,dy \\
    &\lesssim \int_Q\phi(y,|f(y)|)\,dy.
  \end{align*}
  
  Now we prove that $\ref{itm:Jensen_b}\Rightarrow\ref{itm:Jensen_c}$. Suppose $\norm{g}_{\phi^*}\leq1$. Then 
we have the following:
  \begin{align*}
    \int_Q\phi^*
    \left(x,\dashint_Q\abs{g(z)}dz\right)\,dx
    &\leq\int_{Q}(\phi^*)'\left(x,\dashint_{Q}\abs{g(z)}dz\right)\dashint_{Q}\abs{g(z)}\,dz\,dx
    \\
    &=\int_Q\Bigg(\dashint_Q(\phi^*)'\left(y,\dashint_{Q}\abs{g(z)}dz\right)dy\Bigg)\abs{g(x)}\,dx
    \\
    &\leq \tfrac{1}{2p'}\int_Q\phi\left(x,\dashint_{Q}(\phi^*)'\left(y,\dashint_Q\abs{g(z)}dz\right)dy\right)\,dx
    \\
    &\quad + c\int_Q\phi^*(x,\abs{g(x)})\,dx
    \\
    &\leq \tfrac{1}{2p'}\int_Q\phi\left(x,(\phi^*)'\left(x,\dashint_Q\abs{g(z)}dz\right)\right)\,dx
    \\
    &\quad + c\int_Q\phi^*(x,\abs{g(x)})\,dx
    \\
    &\leq\tfrac{1}{2}\int_{Q}\phi^*\left(x,\dashint_{Q}\abs{g(z)}dz\right)dx+c\int_Q\phi^*(x,\abs{g(x)})\,dx.
  \end{align*}
  This implies that
  \begin{align*}
    \int_Q\phi^*\left(x,\dashint_Q\abs{g(z)}dz\right)\,dx\leq c \int_{Q}\phi^*(x,\abs{g(y)})\,dx.
  \end{align*}
 Note that we have applied the generalized Jensen's inequality in  \ref{itm:Jensen_b}  to the function $(\phi^{\ast})'\left(\cdot,\dashint_Q\abs{g(z)}\,dz\right)$ with respect to $\phi$. For this, we need to justify that $\|(\phi^{\ast})'(\cdot,\dashint_Q\abs{g(z)}\,dz)\|_{\phi}\leq c$ uniformly for all cubes $Q$. Since \ref{itm:Jensen_b} implies $\phi\in\mathcal{A}$ and hence $\phi^\ast\in\mathcal{A}$, as mentioned above, this is equivalent to the uniform boundedness of the averaging operators for single cubes. Thus, since $\norm{g}_{\phi^*}\leq1$, $\int_{Q}\phi^*\left(x,\dashint_Q\abs{g(z)}\,dz\right)\,dx\leq c$ for some positive constant $c$, uniformly for all $Q$. On the other hand, using $\phi(x,(\phi^{\ast})'(x,t))\leq p'\phi^*(x,t)$ and the unit ball property, we obtain the desired inequality. 

  Finally, using the fact that $(\phi^*)^*=\phi$, we also obtain that $\ref{itm:Jensen_c}\Rightarrow \ref{itm:Jensen_b}$. Therefore, the proof is completed.
\end{proof}
\begin{remark}
	We subsequently apply the generalized Jensen's inequality (Theorem \ref{thm:Jensen}) to balls. For completeness, we briefly verify that the inequality remains valid in this setting. Here, $\alpha B$ denotes the ball concentric with $B$ and with radius $\alpha$ times that of $B$. Assume the inequality holds for cubes. Without loss of generality, let $f \in L^\phi(\RRn)$ be supported on $B$ with $\norm{f}_\phi\leq1$, and let $Q$ be a cube such that $B \subset Q \subset \sqrt{n} B$. Then
	\begin{align*}
		\dashint_{B}\phi\left(x,\dashint_{B} |f(y)|\,dy\right)\, dx &\lesssim
		\dashint_{Q}\phi\left(x,\dashint_{Q} |f(y)|\, dy\right)\, dx \\
		&\lesssim \dashint_{Q}\phi(x,\abs{f(x)})\,dx \le \dashint_{B} \phi(x,\abs{f(x)})\,dx.
	\end{align*}
\end{remark}
\subsection{Examples}
\label{sec:examples}

In this section we provide examples of the modulating coefficient~$a$ such that the double phase model $\phi(x,t)=\tfrac 1p t^p+\tfrac 1qa(x) t^q$ satisfies our Muckenhoupt condition $\phi \in \mathcal{A}$.
\begin{example}\label{exa:Aq}
  Let $\phi(x,t)=\tfrac 1p t^p+  \tfrac 1q a(x) t^q$ with $1 <p < q < \infty$ and $a \in A_q$, where $A_q$ is the standard Muckenhoupt class~\eqref{eq:muckenhoupt-classical}. Note that this  also includes unbounded coefficients, e.g. $w(x) = \abs{x}^\alpha$ with $-n < \alpha < (q-1)n$. We claim that $\phi \in \mathcal{A}$. Due to Theorem~\ref{thm:Jensen} it is enough to verify the generalized Jensen's inequality.  Let $f\in L^{\phi}({\RRn})$. Then by H\"older inequality,
  \begin{align*}
    \lefteqn{\dashint_{Q}\phi\left(x,\dashint_{Q}\abs{f(y)}\,dy\right)\,dx} & \qquad
    \\
    &= \frac 1p \left(\dashint_{Q}\abs{f(y)}\,dy\right)^p + \frac 1q \dashint_Q a(x)\, dx \left(\dashint_{Q}\abs{f(y)}\,dy\right)^q
    \\
    &\leq \frac 1p \dashint_{Q}\abs{f(x)}^p\,dx + \frac 1q \dashint_{Q}a(x)\,dx \left(\dashint_{Q}a(x)^{-\frac{q'}{q}}\,dx\right)^{\frac{q}{q'}}
    \left(\dashint_{Q}a(x)\abs{f(x)}^q\,dx\right)
    \\
    &\leq [\omega]_{A_q} \dashint_{Q}\phi(x,\abs{f(x)})\,dx.
  \end{align*}
  Note that, compared to Jensen's inequality in Theorem~\refeq{thm:Jensen}, we do not need the restriction that $\norm{f}_\phi \leq 1$ in this example.
\end{example}
The following example is one considered in \cite[Proposition 3.1.]{CoMi1}, (see also \cite[Corollary 7.2.3]{HastoHarjulehto2019}).
\begin{example}
  \label{exa:hoelder}
  Let  $\phi(x,t)= \frac{1}{p}t^p+ \frac{1}{q}a(x) t^q$ with $1 <p < q < \infty$  and a non-negative function $a\in L^\infty(\RRn)$ such that 
  \begin{equation}\label{eq:large_hoelder}
  	a(x)\lesssim a(y)+\abs{x-y}^\alpha
  \end{equation}
  for all $x,y\in Q$ with $\abs{Q}\leq 1$, provided that $\tfrac{q}{p}\leq 1+\tfrac{\alpha}{n}$ with $\alpha>0$. Note that if $\alpha \le 1$, then $a\in C^{0,\alpha}(\RRn)$ implies \eqref{eq:large_hoelder}. We claim that if $a$ satisfies \eqref{eq:large_hoelder}, then $\phi \in \mathcal{A}$. Due to Theorem~\ref{thm:Jensen} it suffices to prove the generalized Jensen's inequality~\eqref{eq:generalized-jensen-phi}. The pointwise inequality leading to the generalized Jensen inequality is proved in \cite[Proposition 3.1]{CoMi1} for the case of balls with radius $R \le 2$. The estimate for all balls can be found in~\cite[Section~7.2]{HastoHarjulehto2019}. In fact, a stronger pointwise estimate holds, see \cite[Theorem~4.3.2]{HastoHarjulehto2019}, which gives the generalized Jensen's inequality after integration. Note that such a pointwise estimate cannot hold in the general case~$\phi \in \mathcal{A}$.

  For the sake of completeness we provide a short proof. Let $Q \subset \RRn$ be a cube and $\norm{f}_\phi \leq 1$. We begin with the case $\abs{Q} \leq 1$. Then $R:=\frac 12\diameter(Q) \leq \frac 12\sqrt{n}$. Using $\norm{f}_p \leq \norm{f}_\phi \leq 1$ and $\alpha+n-\tfrac{nq}{p}\geq 0$ we estimate
  \begin{align*}
    \lefteqn{\dashint_{Q}\phi\bigg(x,\dashint_{Q}\abs{f}\,dy\bigg)\,dx
      =\frac{1}{p}\bigg(\dashint_{Q}\abs{f}\,dy\bigg)^p + \frac{1}{q} \bigg(\dashint_{Q}\abs{f(y)}\,dy\bigg)^q  \dashint_{Q}  a(x)\,dx} \quad &
    \\
    &\leq\frac{1}{p}\dashint_{Q}\abs{f}^p\,dx +\frac{1}{q} \big(\inf_{y\in Q}a(y) + c R^\alpha\big) \bigg(\dashint_{Q}\abs{f(y)}\,dy\bigg)^q 
    \\
    &\lesssim \frac{1}{p} \dashint_{Q}\abs{f}^p\,dx +\frac{1}{q} \dashint_{Q}\! a(y)\abs{f(y)}^q\,dy
    +\frac{R^\alpha}{q} \bigg(\dashint_{Q}\abs{f}^p\,dy\bigg)^{\frac{q}{p}} 
    \\
    &\lesssim \frac{1}{p} \dashint_{Q}\abs{f}^p\,dx+\frac{1}{q}\dashint_{Q} a(y)\abs{f(y)}^q\,dy
    +\frac{1}{q} R^{\alpha+n-\frac{nq}{p}} \left(\dashint_{Q}\abs{f}^p\,dy\right)
    \\
    &
    \lesssim \dashint_{Q}\tfrac 1p\abs{f}^p+ a(x) \tfrac{1}{q}\abs{f}^q\,dx.
  \end{align*}
  Let us now consider the case $\abs{Q}>1$.
  \begin{align*}
    \lefteqn{\dashint_{Q}\phi\left(x,\dashint_{Q}\abs{f}\,dy\right)\,dx
			=\frac{1}{p} \left(\dashint_{Q}\abs{f}\,dy\right)^p + \frac{1}{q}\left(\dashint_{Q}\abs{f}\,dy\right)^q\dashint_{Q} a(x)\,dx} \qquad &
    \\
    &\leq\frac{1}{p}\dashint_{Q}\abs{f}^p\,dx +\frac{1}{q} \norm{a}_\infty \left(\dashint_{Q}\abs{f}\,dy\right)^q
    \\
    &\leq\frac{1}{p}\dashint_{Q}\abs{f}^p\,dx +\frac{1}{q} \norm{a}_\infty \left(\dashint_{Q}\abs{f}^p\,dy\right)^{\frac qp}
    \\
    &\leq\frac{1}{p}\dashint_{Q}\abs{f}^p\,dx \Bigg( 1+ \frac{p}{q} \norm{a}_\infty \abs{Q}^{1-\frac qp} \bigg( \int_Q \abs{f}^p\,dx\bigg)^{\frac qp -1} \Bigg)
    \\
    &\lesssim \frac{1}{p}\dashint_{Q}\abs{f}^p\,dx.
  \end{align*}
  This completes the proof of the generalized Jensen's inequality. Thus, $\phi \in \mathcal{A}$.
\end{example}

\begin{example}
  \label{exa:min-max}
  Let $\phi_j(x,t) = \frac 1p t^p + \frac1q a_j(x) t^q$ with $\phi_j \in \mathcal{A}$ for $j=1,2$. Let
  \begin{alignat*}{2}
    \phi_{\mathrm{sum}} (x,t) &:= \tfrac 1p t^p +  \tfrac 1q (a_1(x)+ a_2(x))t^q,
    \\
    \phi_{\max}(x,t) &:= \tfrac 1p t^p + \tfrac 1q\max \set{a_1(x), a_2(x)}t^q &&= \max \set{ \phi_1(x,t), \phi_2(x,t)},
    \\
    \phi_{\min}(x,t) &:= \tfrac 1p t^p + \tfrac 1q\min \set{a_1(x), a_2(x)}t^q&&=\min \set{ \phi_1(x,t), \phi_2(x,t)}.
  \end{alignat*}
  We claim that $\phi_{\mathrm{sum}}, \phi_{\max}, \phi_{\min} \in \mathcal{A}$. For this we observe that
  \begin{alignat*}{4}
    \norm{f}_{\phi_{\max}} &\leq
    \norm{f}_{\phi_{\mathrm{sum}}} &&\leq
    \norm{f}_{\phi_1+\phi_2}  &&\leq \norm{f}_{\phi_1} + \norm{f}_{\phi_2},
    \\
    \norm{f}_{(\phi_1+\phi_2)^*} &\leq
    \norm{f}_{\phi_{\mathrm{sum}}^*} &&\leq
    \norm{f}_{\phi_{\max}^*} &&\leq \min\set{ \norm{f}_{\phi_1^*}, \norm{f}_{\phi_1^*}}.
  \end{alignat*}
  This and the proof of Remark~\ref{rem:A-additive} imply
  \begin{align*}
    \max \bigset{[\phi_{\mathrm{max}}]_{\mathcal{A}},       [\phi_{\mathrm{sum}}]_{\mathcal{A}}}
    &\le [\phi_1]_{\mathcal{A}} + [\phi_2]_{\mathcal{A}}.
  \end{align*}
  Since $\phi_{\min}$ is convex it follows by~\cite[Theorem~1.3]{TikhIoff} that $\phi_{\min}^* = \max_{j=1,2} \phi_j^*$.
  This, the estimates
  \begin{alignat*}{3}
    \norm{f}_{\phi_{\min}} &\leq
    \min_{j=1,2} \norm{f}_{\phi_j},
    \\
    \norm{f}_{\phi_{\min}^*} &=
    \norm{f}_{\max_j \phi_j^*} &&\leq
    \norm{f}_{\phi_1^* + \phi_2^*} &&\leq
    \norm{f}_{\phi_1^*} + \norm{f}_{\phi_2^*},
  \end{alignat*}
  and the proof in Remark~\ref{rem:A-additive} again implies that
  \begin{align*}
    [\phi_{\min}]_{\mathcal{A}} &\le [\phi_1]_{\mathcal{A}} + [\phi_2]_{\mathcal{A}}.
  \end{align*}
  Now, using $a_1\in A_q$ as in Example~\ref{exa:Aq} and $a_2(x) = (\max\set{|x|,0})^\alpha$ with $\frac qp \leq 1 + \frac \alpha n$  as in Example~\ref{exa:hoelder} we get new examples that cannot be covered by  Example~\ref{exa:Aq} or  Example~\ref{exa:hoelder} alone. In particular, $\phi_{\min}$ is of interest here, since it can be zero on a large set.
\end{example}

\begin{example}
  \label{exa:power}
  Let $\phi(x,t) = \frac 1p t^p +  \frac 1q a(x) t^q$, with $1 < p < q < \infty$ and $a(x) = \min \set{\abs{x}^\alpha,1}$ with $\alpha> -n$. We claim that $\phi\in\mathcal{A}$. If $\alpha \in (-n,n(q-1))$, then $a \in A_q$ by \cite[Exercise 7.1.8]{Grafakos2014classical}. Hence, we can apply 
  Example~\ref{exa:Aq} and Example~\ref{exa:min-max} to obtain $\phi \in \mathcal{A}$. If on the other hand $\alpha \geq n(\frac qp -1)$, exactly the same steps as in Example~\ref{exa:hoelder} imply that $\phi \in \mathcal{A}$, provided we verify the condition $a(x) \lesssim a(y) + \abs{x-y}^\alpha$. Obviously, $\min\{\abs{x}, 1\} \leq \min\{\abs{y}, 1\} + \abs{x-y}$. Raising this to the power $\alpha$ and using the inequality $(u+v)^\alpha \leq 2^{\alpha-1}(u^\alpha + v^\alpha)$ yields
  \[
  a(x) = \big(\min\{\abs{x}, 1\}\big)^\alpha 
  \leq \big(\min\{\abs{y}, 1\} + \abs{x-y}\big)^\alpha 
  \leq 2^{\alpha-1} \big(a(y) + \abs{x-y}^\alpha\big).
  \]
  Using the last inequality with the fact that $a \in L^\infty(\RRn)$, confirms the claim. This example is important because it shows that we can cover cases outside of the range $\frac{q}{p}\geq 1+\frac{\alpha}{n}.$
\end{example}

\subsection{Improved generalized Jensen's inequality}

We now prove an improved version of the generalized Jensen's inequality, which is essential for obtaining the Sobolev-\Poincare inequality. In this step no additional property of the double phase model is used. Just the validity of the generalized Jensen's inequality is enough.  
\begin{definition}[$A_\infty(Q)$ Condition]
	Let $Q$ be a cube. A weight $w$ is said to satisfy the $A_\infty(Q)$ condition if for any $\alpha \in (0, 1)$, there exists $\beta \in (0, 1)$ such that for any cube $Q' \subset Q$ and for any measurable subset $E \subset Q'$, the condition $\lvert E \rvert \geq \alpha \lvert Q' \rvert$ implies $w(E) \geq \beta w(Q')$, where $w(E)\coloneq\int_Ew(x)\,dx$.
\end{definition}

\begin{remark}\label{rem:equivalent_characterizations}
	There are several equivalent characterizations of the class $A_{\infty}(Q)$. In particular, the following statements are equivalent:
	\begin{enumerate}
		\item[(a)] $w \in A_{\infty}(Q)$.
		\item[(b)] For any $\alpha' \in (0, 1)$, there exists $\beta' \in (0, 1)$ such that for any cube $Q' \subset Q$ and any measurable subset $E \subset Q'$, if $\lvert E \rvert \leq \alpha' \lvert Q' \rvert$, then $w(E) \leq \beta' w(Q')$.
		\item[(c)] There exist positive constants $C$ and $\varepsilon$ such that for any cube $Q'\subset Q$,
		\begin{align}\label{eq:RH}
			\left( \frac{1}{|Q'|} \int_{Q'} w(x)^{1+\varepsilon} \, dx \right)^{\frac{1}{1+\varepsilon}} \leq \frac{C}{|Q'|} \int_{Q'} w(x) \, dx.
		\end{align}
	\end{enumerate}
\end{remark}
See \cites[Theorem 7.3.3]{Grafakos2014classical}[see also][p.402, Lemma 2.5]{RubGar}, for an extensive list of equivalent conditions about $A_{\infty}$ class. Note that \cite[Theorem 7.3.3]{Grafakos2014classical} deals with the global case, so it considers the case $A_{\infty}(\RRn)$, but it is easy to see that the local case can be done exactly in the same same way. The local formulation is necessary because the derivation of the $A_\infty$ condition relies on the generalized Jensen's inequality, which is applicable only when the function satisfies a norm bound (specifically, norm $\leq1$ or bounded by a positive constant). This imposes the coupling constraint $0 < t \leq \frac{1}{\|\indicator_{Q}\|_\phi}$. To claim that $\phi(\cdot, t) \in A_\infty(\mathbb{R}^n)$ for a \textit{fixed} $t$, this parameter would need to satisfy the constraint for every cube $Q\subset\RRn$. Since the upper bound depends on $Q$, no single fixed $t$ can satisfy this uniformly. Therefore, the property must be stated locally: $\phi(\cdot, t) \in A_\infty(Q)$ precisely when $t$ is adapted to the specific cube $Q$ in a way that $0 < t \leq \frac{1}{\|\indicator_{Q}\|_\phi}$. Note also that the constants in \eqref{eq:RH} depend only on $\alpha$ and $\beta$ from the definition of $A_\infty$, and on the dimension.

\begin{theorem}\label{thm:A_loc_with_t}
  Let $\phi \in \mathcal{A}$, where $\phi(x,t)=\tfrac 1p t^p+ \tfrac 1q a(x)t^q$, and let $Q \subset \RRn$ be a cube. If $0<t\leq\frac{1}{\norm{\indicator_{Q}}_\phi}$ then there exist positive constants $C$ and $\varepsilon$ depending only on $[\phi]_{\mathcal{A}},\,p,\,q$ and $n$ such that
  \begin{align}\label{eq:phi_RH}
    \left(\dashint_{Q} \phi(x,t)^{1+\varepsilon} dx\right)^{\frac{1}{1+\varepsilon}} \leq C \dashint_{Q} \phi(x,t) dx.
  \end{align} 
\end{theorem}
\begin{proof}
  
  First we prove that the function $\phi(\cdot,t)$ is of class $A_{\infty}(Q)$ for every fixed $t\in\left(0,\frac{1}{\norm{\indicator_{Q}}_\phi}\right] $. Let $Q'\subset Q$ be any subcube and $E\subset Q'$ any measurable subset such that $|E|/|Q'|\geq \alpha$, where $0<\alpha<1$. For the function $f(x)=t\indicator_{E}(x)$, where $0<t\leq \frac{1}{\norm{\indicator_Q}_{\phi}}$, we have
  \begin{align*}
    \norm{f}_{\phi}=t\norm{\indicator_E}_{\phi}\leq1
  \end{align*}
  and
  \begin{align*}
    \dashint_{Q'}f(y)\,dy=t\cdot\frac{|E|}{|Q'|}\geq\alpha t.
  \end{align*}
  By \eqref{eq:lower-upper1} and generalized Jensen's inequality, Theorem \ref{thm:Jensen}, we have
  \begin{align*}
    \int_{Q'}\phi(x,t)dx &\leq\int_{Q'}\phi\bigg(x,\frac{1}{\alpha}\dashint_{Q'}f(y)\,dy\bigg)dx
    \\
    &\leq\bigg(\frac{1}{\alpha}\bigg)^q \int_{Q'}\phi\big(x,\dashint_{Q'}f(y)\,dy\bigg)dx
    \\
    &\leq C \bigg(\frac{1}{\alpha}\bigg)^q \int_{Q'}\phi(x,|f(x)|)dx
    \\
    &= C \bigg(\frac{1}{\alpha}\bigg)^q  \int_{E}\phi(x,t)dx.
  \end{align*}
  Using Remark \ref{rem:equivalent_characterizations}, we conclude that
  \begin{align*}
  	\left(\dashint_{Q'} \phi(x,t)^{1+\varepsilon} \, dx\right)^{\frac{1}{1+\varepsilon}} \leq C \dashint_{Q'} \phi(x,t) \, dx
  \end{align*}
  for any $Q'\subset Q$. Taking $Q'=Q$ completes the proof.
\end{proof} 

\begin{corollary}[Improved generalized Jensen's inequality]\label{lem:highJensen}
Let $\phi \in \mathcal{A}$ with $\phi(x,t)=\tfrac 1p t^p+ \tfrac 1q a(x)t^q$. Assume that $f\in L^{\phi}(\mathbb{R}^n)$ satisfies $\norm{f}_{\phi}\leq 1$. Then there exist constants $s>1$ and $C>0$, depending only on $p$, $q$, and $[\phi]_{\mathcal{A}}$, such that the following inequality holds:
  \begin{align*}
    \Bigg(\dashint_{Q}\phi^s\bigg(y,\dashint_{Q}|f(x)|\,dx\bigg)\,dy\Bigg)^\frac{1}{s}\leq C \dashint_{Q}\phi(y,|f(y)|)\,dy
  \end{align*} 
  for every cube $Q$. 
\end{corollary}
\begin{proof}
  Since $\norm{f}_{\phi}\leq1$, using H\"older's inequality and the fact that $\phi\in\mathcal{A}$ we get that 
  \begin{align*}
    \frac{1}{|Q|}\int_{Q}|f(y)|\,dy\leq 2\frac{1}{|Q|}\norm{f}_{\phi}\norm{\indicator_{Q}}_{\phi^*}\leq\frac{2[\phi]_{\mathcal{A}}}{\norm{\indicator_{Q}}_{\phi}}.
  \end{align*} 
  Now using the $\Delta_{2}$-condition for $\phi$, Theorem \ref{thm:A_loc_with_t} and Theorem \ref{thm:Jensen} we get 
  \begin{align*}
    \Bigg(\dashint_{Q}\phi^s\left(y,\dashint_{Q}|f(x)|\,dx\right)\,dy\Bigg)^\frac{1}{s}\leq c\dashint_{Q}\phi\left(y,\dashint_{Q}|f(x)|\,dx\right)\,dy\leq c \dashint_{Q}\phi(y,|f(y)|)\,dy.
  \end{align*} 
Note that $s=1+\varepsilon$, where $\varepsilon$ is the same as in Theorem \ref{thm:A_loc_with_t}.
\end{proof}

Before we proceed to the next lemma we introduce some useful notation. For $\theta\in(0,\infty)$, let us define the function $\psi_{(\theta)}$ by
\begin{align*}
	\psi_{(\theta)}(x,t)\coloneq\phi(x,t^{\theta})=\tfrac{1}{p}t^{p\theta}+a(x)\tfrac{1}{q}t^{q\theta}.
\end{align*} 
\begin{lemma}\label{lem:Improved_inside}
 Let $\phi \in \mathcal{A}$ with $\phi(x,t)=\tfrac 1p t^p+ \tfrac 1q a(x)t^q$. Assume that $f\in L^{\phi}(\mathbb{R}^n)$ satisfies $\norm{f}_{\phi}\leq 1$. Then there exist $s>1$ and $c>0$ depending only on $[\phi]_{\mathcal{A}}$, $p$, $q$ and $n$, such that the following inequality holds:
  \begin{align}\label{eq:left-open-musielak}
    \dashint_{Q}\phi\bigg(y,\bigg(\dashint_{Q}|f(x)|^s\,dx\bigg)^{\frac{1}{s}}\bigg)\,dy\leq c \dashint_{Q}\phi(y,|f(y)|)\,dy
  \end{align} 
\end{lemma}
\begin{proof}
  If $\phi\in\mathcal{A}$ then using Theorem \ref{thm:Jensen} we can easily show that $\phi\in\mathcal{A}_{\text{glob}}$. By~\cite[Theorem 5.4.15]{DHHR}, there exists\footnote{%
    Theorem~5.4.15 in \cite{DHHR} is formulated for proper N-functions, which ensures that $L^\phi$ is a Banach function space. This requires that $\indicator_E \in L^\phi \cap L^{\phi^*}$ for measurable set with finite measure. However, for the theory of Banach function spaces it is enough assume this property for every compact subset of~$\RRn$. See also the comments in~\cite{Ler17}.
  } a $\theta_0\in (0,1)$, depending only on the $\mathcal{A}_{\text{glob}}$-constant of $\phi$ and the $\Delta_{2}$-constant of its conjugate $\phi^*$, such that $\psi_{(\theta)}\in \mathcal{A}_{\text{glob}}$ for all $\theta\geq \theta_{0}$. Since $\mathcal{A}_{\text{glob}}$ implies $\mathcal{A}$, we have $\psi_{(\theta)}\in\mathcal{A}$, and by Theorem~\ref{thm:Jensen}, $\psi_{(\theta)}$ satisfies the generalized Jensen's inequality. 
  On the other hand, from the definition of $\psi_{(\theta)}$, we have that $\norm{g}_{	\psi_{(\theta)}}^{\theta}=\norm{\abs{g}^\theta}_\phi$.
  Hence $\norm{g}_{	\psi_{(\theta)}}\leq 1$ if and only if $\norm{f}_\phi
  \leq 1$, where $g\coloneq \abs{f}^{1/\theta}$. Now we can use Theorem~\ref{thm:Jensen} for $\psi_{(\theta)}$ and for $g$. Then we get that 
  \begin{align*}
    \dashint_{Q}\psi_{(\theta)}\bigg(y,\dashint_{Q}|g(x)|\,dx\bigg)\,dy\leq c \dashint_{Q}\psi_{(\theta)}(y,|g(y)|)\,dy.
  \end{align*}
  By the definition of $\psi_{(\theta)}$ and the fact that $g= \abs{f}^{1/\theta}$ we get the desired inequality \eqref{eq:left-open-musielak}, where $s\coloneq\frac{1}{\theta}$.
\end{proof}

\subsection{Boundedness of the maximal function}

In this subsection we prove the boundedness of the maximal function. For this we use the well-known dyadic maximal function, which essentially simplifies some technical steps. Note that the results of this sections are derived solely from the generalized Jensen's inequality.

We say that a family of cubes $\mathcal{D}$ is a dyadic grid if:
\begin{enumerate}
	\item[(a)] for any $Q\in \mathcal{D}$ its sidelength $\ell_Q$ is of the form $2^{k}$, $k\in \mathbb{Z}$;
	\item[(b)] $Q_1\cap Q_2\in\{Q_1,Q_2,\emptyset\}$ for any $Q_1,Q_2\in \mathcal{D}$;
	\item[(c)] for every $k\in \mathbb{Z}$, the cubes of a fixed sidelength $2^{k}$ form a partition of $\mathbb{R}^n$.
\end{enumerate}
Given a dyadic grid $\mathcal{D}$ and any locally integrable function $f$, we consider the associated dyadic maximal function $M^{\mathcal{D}}f$ defined by
\begin{align*}
	M^{\mathcal{D}}f(x) = \sup_{x \in Q, Q \in \mathcal{D}} \frac{1}{|Q|} \int_Q |f(y)| \, dy.
\end{align*}
where the supremum is taken over all $Q\in\mathcal{D}$ containing $x$.
It is clear that $M^{\mathcal{D}}f\leq Mf.$ However, this inequality can be reversed, in a sense, as the following lemma shows (its proof can be found in \cite[Lemma 2.5]{HLP}).
\begin{lemma}\label{lem:dyadicMax}
  There are $3^{n}$ dyadic grids $\mathcal{D}_{\alpha}$ such that for every cube $Q\subset\RRn$, there exists a cube $Q_{\alpha}\in \mathcal{D}_{\alpha}$ such that $Q\subset Q_{\alpha}$ and $|Q_{\alpha}|\leq 6^{n}|Q|.$
\end{lemma}

We obtain from this lemma that for all $x\in \RRn$,
\begin{equation}\label{revested dyadic}
  Mf(x)\leq 6^{n}\sum_{\alpha=1}^{3^{n}}M^{\mathcal{D}_{\alpha}}f(x).
\end{equation}

The following lemma is a standard version of the Calder\'{o}n-Zygmund decomposition (see, e.g. \cite[Theorem 5.3.1]{Grafakos2014classical}).

\begin{lemma}\label{lem:dyadic_dec}
  Suppose $\mathcal{D}$ is a dyadic grid. Let $f\in L^{p}(\RRn)$ with $1\leq p<\infty$, and let $\gamma>1.$ Assume that for $k\in\mathbb Z$,
\begin{align*}
\Omega_{k}:=\{x\in   \RRn:\,M^{\mathcal{D}}f(x)>\gamma^{k}\}\neq\emptyset.
\end{align*}
Then $\Omega_{k}$ can be written as a union of pairwise disjoint  maximal cubes $Q_{k,l}\in \mathcal{D}$, $l\in \mathbb N$, satisfying
  \begin{align}\label{sparse}
    \begin{aligned}
     & \gamma^{k}<\frac{1}{|Q_{k,l}|}\int_{Q_{k,l}}|f(x)|dx \leq 2^{n}\gamma^{k},\\
     & |Q_{k,l} \cap\Omega_{k+m}|  \leq 2^{n}(1/\gamma)^{m}|Q_{k,l}|
      \quad \text{for }\ m\in\mathbb{Z}_{+}.
    \end{aligned}
  \end{align}
\end{lemma}
We now state and prove our first main result. Recall that 
	$\psi_{(s)}(x,t)\coloneq\phi(x,t^{s}).$
\begin{theorem}\label{thm:M_bounded}
	Let $\phi(x,t)=\tfrac 1p t^p+ a(x)\tfrac 1qt^q$ with $1<p<q < \infty$ and $a\,:\, \RRn \to [0,\infty)$ be measurable. Then the following assertions are equivalent: 
	\begin{enumerate}
  \item\label{itm:M_bounded_Muck} $\phi \in \mathcal{A}$.
  \item \label{itm:M_bounded_modular} The inequality
		\begin{align*}
			\int_{\RRn}\phi(x,Mf(x))\,dx\lesssim\int_{\RRn}\phi(x,\abs{f(x)})\,dx
		\end{align*}
		holds for all $f\in L^\phi(\RRn)$ with $\norm{f}_\phi\leq1$.
  \item\label{itm:M_bounded_norm} The maximal operator is bounded on $L^\phi(\RRn)$.
  \item \label{itm:M_bounded_modular_dual} The inequality
		\begin{align*}
			\int_{\RRn}\phi^*(x,Mg(x))\,dx\lesssim\int_{\RRn}\phi^*(x,\abs{g(x)})\,dx
		\end{align*}
		holds for all $g\in L^{\phi^*}(\RRn)$ with $\norm{g}_{\phi^*}\leq1$.
  \item\label{itm:M_bounded_norm_dual} The maximal operator is bounded on $L^{\phi^*}(\RRn)$.
  \item \label{itm:M_bounded_left_open} There exists $0 < s < 1$ such that the maximal operator is bounded on $L^{\psi_{(s)}}(\RRn)$.
	\end{enumerate}
\end{theorem}
\begin{proof}
	
  The proof of the theorem will be carried out in the following  order: 	$\ref{itm:M_bounded_Muck}\Rightarrow\ref{itm:M_bounded_modular}\Rightarrow\ref{itm:M_bounded_norm}\Rightarrow\ref{itm:M_bounded_modular_dual}\Rightarrow\ref{itm:M_bounded_norm_dual}\Rightarrow\ref{itm:M_bounded_left_open}\Rightarrow\refeq{itm:M_bounded_Muck}$.
  First we prove the case $\ref{itm:M_bounded_Muck}\Rightarrow\ref{itm:M_bounded_modular}.$ The argument for $\ref{itm:M_bounded_Muck}\Rightarrow\ref{itm:M_bounded_modular}$ is inspired by that of \cite[Theorem 5.5.12]{DHHR}, but our modifications allow for a direct proof of the modular inequality. Due to Lemma~\ref{lem:dyadicMax}, it suffices to prove the theorem for the dyadic maximal function $M^{\mathcal{D}}$ associated with an arbitrary dyadic grid $\mathcal{D}$. 
  Let $f\in L^{\phi}(\RRn)$ be such that $\int_{\RRn}\phi(x,\abs{f(x)})\,dx\leq 1$. Using Fubini's theorem we have that
  \begin{align*}
    \int_{\RRn}\phi(x,M^{\mathcal{D}}f(x))\,dx
    &= \int_{0}^{\infty}\int_{\RRn}\phi'(x,\lambda)\indicator_{\{M^{\mathcal{D}}f(x)>\lambda\}}\,dx\,d\lambda \\
    &= \sum_{k\in \mathbb{Z}}\int_{2^{k+1}}^{2^{k+2}}\int_{\RRn}\phi'(x,\lambda)\indicator_{\{M^{\mathcal{D}}f(x)>\lambda\}}\,dx\,d\lambda \\
    &\lesssim \sum_{k\in \mathbb{Z}}\int_{\RRn}\phi(x,2^{k+2})\indicator_{\{M^{\mathcal{D}}f(x)>2^{k+1}\}}\,dx.
  \end{align*}  
  For each $k \in \mathbb{Z}$, we decompose $f$ by setting
  \begin{align*}
    f_{0,k}(x) &\coloneq f(x)\indicator_{\{\abs{f(x)} > 2^{k+1}\}}, \\
    f_{1,k}(x) &\coloneq f(x)\indicator_{\{\abs{f(x)} \leq 2^{k+1}\}}.
  \end{align*}
  Since $M^{\mathcal{D}}f \le M^{\mathcal{D}}f_{0,k} + M^{\mathcal{D}}f_{1,k}$, we have the inclusion
  \begin{align*}
    \{M^{\mathcal{D}}f>2^{k+1}\} \subset \{M^{\mathcal{D}}f_{0,k} > 2^{k}\} \cup \{M^{\mathcal{D}}f_{1,k}>2^{k}\}.
  \end{align*}
  From this and  the $\Delta_{2}$ condition, we obtain
  \begin{align*}
    \int_{\RRn}\phi(x,M^{\mathcal{D}}f(x))\,dx \lesssim \sum_{j=0}^{1}\sum_{k\in \mathbb{Z}}\int_{\RRn}\phi(x,2^{k})\indicator_{\{M^{\mathcal{D}}f_{j,k}>2^{k}\}}(x)\,dx.
  \end{align*}
  Due to Lemma \ref{lem:dyadic_dec} for each $k\in \mathbb{Z}$ and $j\in\{0,1\}$, the set $\{M^{\mathcal{D}}f_{j,k}>2^{k}\}$ can be written  as a union of pairwise disjoint maximal dyadic cubes $\{Q_{j,k,l}\} =: \mathcal{Q}(j,k)$ such that for each $j$
  \begin{align}\label{fjk}
    2^{k}<\dashint_{Q_{j,k,l}} |f_{j,k}(x)|\, dx \leq 2^{n+k}.
  \end{align}
  Using this we get
  \begin{align*}
    \int_{\RRn}\phi(x,2^{k})\indicator_{\{M^{\mathcal{D}}f_{j,k}>2^{k}\}}\,dx = \sum_{Q\in \mathcal{Q}(j,k)}|Q|(M_Q\phi)(2^{k})
  \end{align*}
 As shown in the proof of Lemma~\ref{lem:Improved_inside}, if $\phi$ satisfies property~\ref{itm:M_bounded_Muck}, then $\psi_{(s)}$ satisfies the generalized Jensen's inequality for all $s\geq s_0$, for some $s_0\in(0,1)$. Let us choose $s_0$ such that the generalized Jensen's inequality holds for $\psi_{(s_0)}$ and $s_0 \in (1/p, 1)$. Then, letting $s_{1}>1$, we get
  \begin{align*}
    \int_{\RRn}\phi(x,M^{\mathcal{D}}f(x))\,dx
    &\lesssim \sum_{j=0}^{1}\sum_{k\in \mathbb{Z}}\sum_{Q\in \mathcal{Q}(j,k)}|Q|(M_Q\phi)(2^{k}) \\
    & = \sum_{j=0}^{1}\sum_{k\in \mathbb{Z}}\sum_{Q\in \mathcal{Q}(j,k)}|Q|(M_Q\psi_{(s_{j})})(2^{k/s_{j}}).
  \end{align*}
  The definition of $f_{j,k}$ and $1/p <s_{0}<1<s_{1}$ imply that
  \begin{align*}
    0\leq\abs{f_{j,k}}^{1-1/s_{j}}2^{(k+1)(1/s_{j}-1)}\indicator_{\{f_{j,k}\neq0\}}\leq1,
    \quad  \text{for }\ j \in \{0,1\}\ \text{ and } \ k\in \mathbb{Z}.
  \end{align*} 
By the convexity of $\psi_{(s_{j})}$, it follows that
  \begin{align*}
    &\psi_{(s_{j})}(x, \abs{f_{j,k}(x)}2^{k(1/s_{j}-1)})\\
    &\eqsim \psi_{(s_{j})}(x, \abs{f_{j,k}(x)}2^{(k+1)(1/s_{j}-1)})
    = \psi_{(s_{j})}(x, \abs{f_{j,k}(x)}^{1/s_{j}}\abs{f_{j,k}(x)}^{1-1/s_{j}}2^{(k+1)(1/s_{j}-1)}) \\
    &\leq \psi_{(s_{j})}(x, \abs{f_{j,k}(x)}^{1/s_{j}})\abs{f_{j,k}(x)}^{1-1/s_{j}}2^{(k+1)(1/s_{j}-1)}\indicator_{\{f_{j,k}\neq0\}} \\
    &= \phi(x, \abs{f_{j,k}(x)})\abs{f_{j,k}(x)}^{1-1/s_{j}}2^{(k+1)(1/s_{j}-1)}\indicator_{\{f_{j,k}\neq0\}} \le \phi(x, \abs{f_{j,k}(x)}) .
  \end{align*}
  This implies that $\|2^{k(1/s_{j}-1)} |f_{j,k}|\|_{\psi_{(s_{j})}}\lesssim 1$ as $\|f\|_{\phi}\le 1$ . Using \eqref{fjk}, the generalized Jensen's inequality, Theorem \ref{thm:Jensen}, and the properties of the Calder\'on-Zygmund cubes $Q \in \mathcal{Q}(j,k)$, we obtain that for each $j\in\{0,1\}$,
  \begin{align*}
    \sum_{k\in \mathbb{Z}}&\sum_{Q\in \mathcal{Q}(j,k)}|Q|(M_Q\psi_{(s_{j})})(2^{k/s_{j}}) \\
    &\leq \sum_{k\in \mathbb{Z}}\sum_{Q\in \mathcal{Q}(j,k)}|Q|(M_Q\psi_{(s_{j})})\left(2^{k(1/s_{j}-1)}\dashint_Q |f_{j,k}(y)|\,dy\right) \\
    &\lesssim 
    \sum_{k\in \mathbb{Z}}\sum_{Q\in \mathcal{Q}(j,k)}\int_Q \psi_{(s_{j})}(x, 2^{k(1/s_{j}-1)}|f_{j,k}(x)|)\,dx \\
    &\leq \sum_{k\in \mathbb{Z}}\int_{\RRn}\psi_{(s_{j})}(x,2^{k(1/s_{j}-1)}|f_{j,k}(x)|)\,dx \\
    &\lesssim \sum_{k\in \mathbb{Z}} \int_{\RRn}\phi(x,\abs{f_{j,k}(x)})\abs{f_{j,k}(x)}^{1-1/s_{j}}2^{k(1/s_{j}-1)}\indicator_{\{f_{j,k}\neq0\}}\,dx \\
    &= \int_{\RRn}\phi(x,\abs{f(x)})\abs{f(x)}^{1-1/s_{j}} \left(\sum_{k\in\mathbb{Z}} 2^{k(1/s_{j}-1)}\indicator_{\{f_{j,k}=f, f\neq0\}}\right) dx
  \end{align*}
  By the definition of $f_{j,k}$ and the bounds on $s_j$, we can sum up the series in $k$. For $j=0$ with $1/p<s_0<1$:
  \begin{align*}
    \sum_{k=-\infty}^{\infty}2^{k(1/s_{0}-1)}\indicator_{\{f_{0,k}=f, f\neq0\}}=\sum_{k=-\infty}^{\infty}2^{k(1/s_{0}-1)}\indicator_{\{\abs{f}\ge 2^{k+1}\}}\lesssim  \abs{f}^{1/s_{0}-1}.
  \end{align*}
  And for $j=1$ with $s_1>1$:
  \begin{align*}
    \sum_{k=-\infty}^{\infty}2^{k(1/s_{1}-1)}\indicator_{\{f_{1,k}=f, f\neq0\}}=\sum_{k=-\infty}^{\infty}2^{k(1/s_{1}-1)}\indicator_{\{\abs{f} \leq 2^{k+1}\}}\lesssim \abs{f}^{1/s_{1}-1}.
  \end{align*}
  Combining these estimates, we get
  \begin{align*}
    \int_{\RRn}\phi(x,M^{\mathcal{D}}f(x))\,dx 
    &\lesssim \sum_{j=0}^{1} \int_{\RRn}\phi(x,\abs{f(x)})\abs{f(x)}^{1-1/s_{j}} \abs{f(x)}^{1/s_j-1} \,dx \\
    &\le 2\int_{\RRn}\phi(x,\abs{f(x)})\,dx \le 2.
  \end{align*}
  Using \eqref{revested dyadic}, the convexity of $\phi(x,\cdot)$, and \eqref{eq:lower-upper1}, the proof of $\ref{itm:M_bounded_Muck}\Rightarrow \ref{itm:M_bounded_modular}$ is complete. 
	For the implication $\ref{itm:M_bounded_modular}\Rightarrow\ref{itm:M_bounded_norm}$, we use the fact that the norm boundedness of $M$ is equivalent to the following condition: there exists a constant $C>0$ such that for all non-negative $f\in L^{\phi}(\RRn)$,
	\begin{align*}
		\int_{\RRn}\phi(x,\abs{f(x)})\,dx\leq 1 \quad \implies \quad \int_{\RRn}\phi(x,Mf(x))\,dx\leq C.
	\end{align*} Next, we prove $\ref{itm:M_bounded_norm}\Rightarrow \ref{itm:M_bounded_modular_dual}$. If \ref{itm:M_bounded_norm} holds, then $\phi\in\mathcal{A}$.
	By Theorem~\ref{thm:Jensen}, the generalized Jensen's inequality also holds for the conjugate function $\phi^*$.
	The desired conclusion then follows by applying the same argument used to prove $\ref{itm:M_bounded_Muck}\Rightarrow\ref{itm:M_bounded_modular}$.
The implication $\ref{itm:M_bounded_modular_dual}\Rightarrow\ref{itm:M_bounded_norm_dual}$ is proven analogously to $\ref{itm:M_bounded_modular}\Rightarrow\ref{itm:M_bounded_norm}$. For the implication $\ref{itm:M_bounded_norm_dual}\Rightarrow\ref{itm:M_bounded_left_open}$, note that $\ref{itm:M_bounded_norm_dual}$ implies $\phi\in\mathcal{A}$. Then, as shown in the proof of Lemma~\ref{lem:Improved_inside}, $\psi_{(s)}$ satisfies the generalized Jensen's inequality for some $s \in (0,1)$. Since this is the key property required for the boundedness of the maximal function, the argument for $\ref{itm:M_bounded_Muck}\Rightarrow\ref{itm:M_bounded_modular}$ applies directly to $\psi_{(s)}$. Finally, the implication $\ref{itm:M_bounded_left_open}\Rightarrow\ref{itm:M_bounded_Muck}$ is straightforward. Indeed, the boundedness of $M$ on $L^{\psi_{(s)}}(\mathbb{R}^n)$ for some $s \in (0,1)$ implies its boundedness on $L^{\phi}(\mathbb{R}^n)$ (using the classical Jensen's inequality), which in turn yields $\phi \in \mathcal{A}$.
\end{proof}
\subsection{Density of regular functions}
As a direct implication of the boundedness of the maximal function, we present the density of smooth functions in $L^\phi(\Omega)$ and $W^{1,\phi}(\Omega)$.
Let us begin with the following norm convergence result for standard mollifiers.
\begin{theorem}[Mollification]\label{thm:mollification}
 	Let $\phi(x,t) = \frac{1}{p}t^p + \frac{1}{q}a(x)t^q$ be the double phase function with $\phi \in \mathcal{A}$, and let $\psi$ be a standard mollifier. Define $\psi_\varepsilon(x) = \varepsilon^{-n} \psi(x/\varepsilon)$ for $\varepsilon > 0$. Then for all $f \in L^{\phi}(\mathbb{R}^n)$,
 	\begin{align*}
 		f * \psi_\varepsilon \to f \quad \text{a.e. and in } L^{\phi}(\mathbb{R}^n) \quad \text{as } \varepsilon \to 0.
 	\end{align*}
\end{theorem}

\begin{proof}
 	It is a classical result (see \cite[Theorem 2, p.62]{Stein}) that the convolution is controlled by the maximal operator:
 	\begin{align*}
 		\sup_{\varepsilon > 0} |(f * \psi_\varepsilon)(x)| \leq C M(f)(x) \quad \text{and} \quad \lim_{\varepsilon \to 0} (f * \psi_\varepsilon)(x) = f(x) \quad \text{a.e.}
 	\end{align*}
 	Let $(\varepsilon_n)$ be a sequence of positive numbers such that $\varepsilon_n \to 0$. By setting $f_n \coloneq f * \psi_{\varepsilon_n}$ and $g = C M(f)$, we have $|f_n| \leq g$ a.e. Since $\phi \in \mathcal{A}$, Theorem~\ref{thm:M_bounded} implies $g \in L^\phi(\mathbb{R}^n)$. Consequently, we can apply \cite[Lemma 2.3.16(c)]{DHHR} to the sequence $f_n \to f$, which yields the desired convergence in norm.
\end{proof}

Using Theorem~\ref{thm:mollification}, we obtain the following density result for the Sobolev space. We note that, unlike similar results in the general setting (cf. \cite{HastoHarjulehto2019}), we do not require the assumption (A0) here.

\begin{theorem}\label{thm:density_interior}
 	Let $\phi(x,t) = \frac{1}{p}t^p + \frac{1}{q}a(x)t^q$ and $\phi \in \mathcal{A}$. Then the following holds:
  \begin{enumerate}
  \item $C^\infty(\Omega) \cap W^{1,\phi}(\Omega)$ is dense in $W^{1,\phi}(\Omega)$ and $C^\infty_0(\Omega)$ is dense in $W^{1,\phi}_0(\Omega)$.
  \item If additionally $\Omega$ is a bounded Lipschitz domain, then $C^\infty(\overline{\Omega})$ is dense in $W^{1,\phi}(\Omega)$.
  \end{enumerate}
\end{theorem}
We omit the proof, since it follows exactly the same lines as \cite[Theorems 8.5.2, 9.1.7 and 9.1.8 and Proposition 11.2.3]{DHHR}, utilizing the boundedness of the maximal operator, which controls mollification and allows one to use the extrapolation technique of Rubio de Francia.

\subsection{Sobolev-\Poincare inequality}

For $k\in \mathbb{Z}$ we define the averaging operator $T_k$ over dyadic cubes by 
\begin{align*}
  T_kf:=
  \sum_{\substack{Q \text{ is dyadic}\\ \ell_Q=2^{-k}}} \indicator_Q \dashint_{2Q}\abs{f}dy.
\end{align*}
where $\ell_Q$ is the sidelength of the cube $Q$. Recall that a dyadic cube in $\mathbb{R}^n$ is the set
\[
	[2^k m_1, 2^k(m_1 + 1)) \times \cdots \times [2^k m_n, 2^k(m_n + 1)),
\]
where $k, m_1, \dots, m_n \in \mathbb{Z}$.

\begin{theorem}[Sobolev-\Poincare]\label{thm:sobolev-poincare}
  Let $\phi \in \mathcal{A}$ with $\phi(x,t)=\tfrac 1p t^p+ \tfrac 1q a(x)t^q$. There exist constants $s>1$ and $c > 0$  such that for any $u \in W^{1,1}(B)$ with $\int_{B}\phi(x, |\nabla u|)dx \leq 1$ we have
  \begin{align}\label{eq:Sobolev}
  \dashint_B \phi\left(x, \frac{|u - \mean{u}_{B}|}{r_B}\right) dx
  \le  \left(\dashint_B \phi^s\left(x, \frac{|u - \mean{u}_{B}|}{r_B}\right) dx\right)^{\frac{1}{s}}
    &\leq C
    \dashint_{B} \phi(x, |\nabla u|) \, dx.
  \end{align}
  Moreover, if $u\in W^{1,1}_0(B)$, then we have
    \begin{align}\label{eq:Sobolev1}
     \dashint_B \phi\left(x, \frac{|u|}{r_B}\right) dx \le \left(\dashint_B \phi^s\left(x, \frac{|u|}{r_B}\right) dx\right)^{\frac{1}{s}}
      &\leq C
      \dashint_{B} \phi(x, |\nabla u|) \, dx.
    \end{align}
  
\end{theorem}
\begin{proof}
  Note that the first inequality follows directly from the classical Jensen's inequality since $s>1$. Thus, it suffices to prove the second estimate. To this end, we use the standard Riesz potential estimate: 
  \begin{align}
    \label{eq:poincare-riesz}
    |u(x) - \langle u \rangle_B| \lesssim \int_B \frac{|\nabla u(y)|}{|x - y|^{n-1}} \,dy \lesssim r_B \sum_{k=0}^{\infty} 2^{-k} T_{k+k_0} (\indicator_B|\nabla u|)(x)
  \end{align}
  for all $x\in B$,
  where $k_0\in\mathbb{Z}$ is such that $2^{-k_0-1}\leq r_B\leq2^{-k_0}$. Using this we get that 
  \begin{align*}
    |u(x)-\mean{u}_B| \lesssim r_B
    \sum_{k=k_0}^{\infty}2^{-k+k_0} \sum_{\substack{Q \, \text{dyadic} \\ \ell(Q)=2^{-k}}} \indicator_{Q}(x) \dashint_{2Q} \indicator_{ B}|\nabla u|\,dy.
  \end{align*}
  By the convexity, Corollary~\ref{lem:highJensen} and the $\Delta_2$ condition of $\phi(x, \cdot)$, we have that:
  \begin{align*}
    \lefteqn{\Bigg(\dashint_B \phi^s\left(x, \frac{|u-\mean{u}_B|}{r_B}\right) dx\Bigg)^{\frac{1}{s}}}\qquad&
    \\
    &\lesssim \Bigg(\dashint_B \phi^{s}\Bigg(x,\sum_{k=k_0}^{\infty}2^{-(k-k_0)}
    \sum_{\substack{Q \,\text{dyadic}\\ \ell_Q=2^{-k}}} \indicator_{Q}(x) \dashint_{2Q} \indicator_{B}|\nabla u|\,dy\Bigg)\Bigg)^{\frac{1}{s}}
    \\
    &\leq \sum_{k=k_0}^\infty 2^{-(k-k_0)} \Bigg(\dashint_B \phi^s\bigg(x, \sum_{\substack{Q \,\text{dyadic}\\ \ell_Q=2^{-k}}}\indicator_{Q}(x) \dashint_{2Q} \indicator_{B}|\nabla u|\,dy\bigg)dx \Bigg)^{\frac{1}{s}}
    \\
    &\leq \sum_{k=k_0}^\infty 2^{-(k-k_0)} \sum_{\substack{Q \,\text{dyadic}\\ \ell_Q=2^{-k}}} \Bigg(\frac{|2Q|}{|B|}\Bigg)^{\frac{1}{s}}\Bigg(\dashint_{2Q} \phi^s \left(x,\dashint_{2Q} \indicator_{B}|\nabla u|\,dy\right) dx\Bigg)^{\frac{1}{s}}
    \\
    &\lesssim \sum_{k=k_0}^\infty 2^{-(k-k_0)}\Bigg(\frac{|2Q|}{|B|}\Bigg)^{\frac{1}{s}} \sum_{\substack{Q \,\text{dyadic}\\ \ell_Q=2^{-k}}} \Bigg(\dashint_{2Q} \phi \left(x,\indicator_{B}|\nabla u|\right) dx\Bigg)
    \\
    &\le\sum_{k=k_0}^\infty 2^{-(k-k_0)}\Bigg(\frac{|2Q|}{|B|}\Bigg)^{\frac{1}{s}-1}\dashint_{B}\phi(x,|\nabla u|)\,dx
    \\
    &\leq\sum_{k=k_0}^\infty 2^{-(k-k_0)\frac{n-ns+s}{s}}\dashint_{B}\phi(x,|\nabla u|)\,dx
    \\
    &\lesssim\dashint_{B}\phi(x,|\nabla u|)\,dx.
  \end{align*}
  This implies \eqref{eq:Sobolev}. Moreover, if $u\in W^{1,1}_0(B)$, it follows that
  \begin{align*}
    |u(x)| \lesssim \int_B \frac{|\nabla u(y)|}{|x - y|^{n-1}} \,dy.
  \end{align*}
  Therefore, we obtain \eqref{eq:Sobolev1} in the same way.
\end{proof}
\begin{remark}
  We stated Theorem~\ref{thm:sobolev-poincare} for balls. However, it is true with~$B$ replaced by an $\alpha$-John domain~$\Omega$ with an additional dependence of the constant on~$\alpha$. The important step is that the estimate~\eqref{eq:poincare-riesz} involving the Riesz potential of $\indicator_B \abs{\nabla u}$ also holds on John domains, see e.g. \cite[Lemma~8.2.1]{DHHR} where the definition of an $\alpha$-John domain can also be found. Bounded Lipschitz domains are for example $\alpha$-John domains. The other steps in the proof remain the same. This is in fact very similar to \cite[Theorem~8.2.4]{DHHR}, which inspired our proof of Theorem~\ref{thm:sobolev-poincare}. Note that Corollary~\ref{cor:sobolev-poincare-zeroset} below also holds for $\alpha$-John domains.
\end{remark}
\begin{remark}
	Note that to guarantee the convergence of the series above, we require that $n-ns+s>0$, which implies $s<\tfrac{n}{n-1}$. 
\end{remark}

The same argument as in \cite[Remark 8.2.10]{DHHR} gives us the following result.
\begin{corollary}[Sobolev-\Poincare for functions with a zero set]\label{cor:sobolev-poincare-zeroset}
  Let $\phi \in \mathcal{A}$ with $\phi(x,t)=\tfrac 1p t^p+ \tfrac 1q a(x)t^q$. Under the same assumptions as in Theorem~\ref{thm:sobolev-poincare}, suppose that $u \in W^{1,1}(B)$ vanishes on a set $E \subset B$ 
 of positive measure. 
Then there exists a constant $C>0$, depending on $n$, $p$, $q$ and $[\phi]_{\mathcal{A}}$, 
such that
  \begin{align}\label{eq:Sobolev_no_mean}
    \bigg(\dashint_B \phi^s\bigg(x, \frac{|u|}{r_B}\bigg) \,dx\bigg)^{\frac{1}{s}}
    &\leq C \Big(1+\tfrac{|B|}{|E|}\Big)^q
    \dashint_{B} \phi(x, |\nabla u|) \, dx.
  \end{align}
\end{corollary}

\section{Regularity results}
\label{sec:regularity-results}

In this section, we explore regularity results for minimizers of the double phase functional \eqref{eq:doublephase-energy-intro}, or, equivalently, weak solutions to the double phase equation \eqref{eq:pde}, under the Muckenhoupt condition $\phi \in \mathcal{A}$. In this section, $\phi(x,t)$ is always the double phase function, i.e.
\begin{align}
  \phi(x,t)=\tfrac{1}{p}t^p + \tfrac{1}{q}a(x)t^q.
\end{align}
A function $u\in W^{1,\phi}(\Omega)$ is called a weak solution of 
\begin{equation}\label{eq:PDE_full}
  -\mathrm{div}\left(\frac{\phi'(x,|\nabla u|)}{|\nabla u|} \nabla u \right) = 0
\end{equation}
if
\begin{align}\label{eq:weak_form_full}
  \int_{\Omega}  \frac{\phi'(x,|\nabla u|)}{|\nabla u|} \nabla u  \cdot \nabla \psi \, dx=  \int_{\Omega} \left( |\nabla u|^{p-2}\nabla u + a(x)|\nabla u|^{q-2}\nabla u \right) \cdot \nabla \psi \, dx = 0
\end{align}
for all $\psi \in W_0^{1,\phi}(\Omega)$. Moreover, a function $v\in W^{1,\phi}(\Omega)$ is a weak subsolution (resp. supersolution) of \eqref{eq:PDE_full} if
\begin{align}\label{eq:weak_form_full_subsup}
  \int_{\Omega} \left( |\nabla v|^{p-2}\nabla v + a(x)|\nabla v|^{q-2}\nabla v \right) \cdot \nabla \psi \, dx \le 0 \ (\text{resp.}\ \ge 0) 
\end{align}
for all $\psi \in W_0^{1,\phi}(\Omega)$ with $\psi\ge0$ a.e. in $\Omega$.

\subsection{Caccioppoli inequality and local boundedness}

For any $\lambda\in \RR$, we define the positive part of $u$ above $\lambda$ as 
\begin{align*}
	u_\lambda(x) \coloneqq (u(x)-\lambda)_+  \coloneqq \max\{u(x)-\lambda, 0\},
\end{align*} 
and when $\lambda=0$, we write $u_+ \coloneqq u_0$.
Note that $\nabla u_\lambda = \nabla u$ on the set where $u>\lambda$ and $\nabla u_\lambda = 0$ where $u \le \lambda$.
We begin with a \Caccioppoli~type estimate.

\begin{lemma}[Caccioppoli inequality]\label{lem:Caccio}
	Let $B_r$ and  $B_R$ with $r<R$ be concentric balls such that $B_{R}\subset\Omega$.   If $u \in W^{1,\phi}(\Omega)$ is a weak subsolution of \eqref{eq:PDE_full}, then  we have that for any $\lambda \in\RR$,
	\begin{align*}
		\int_{B_r} \phi(x, \abs{\nabla u_\lambda}) \,dx\leq c \int_{B_R} \phi \left(x, \frac{u_\lambda}{R-r}\right) \, dx ,
	\end{align*}
	where $c>0$ depends only on the dimension $n$, and the exponents $p$ and $q$. 	
\end{lemma}
\begin{proof}
	Let $\eta\in C^\infty_c(B_R)$ be a cut off function such that $0\le \eta \le 1$, $\eta \equiv 1$ on $B_r$ and $|\nabla\eta|\le c(n)/(R-r)$. We take $\psi=\eta^q u_\lambda $ in \eqref{eq:weak_form_full} and use Young's inequality to get
	\begin{align*}
		\int_{B_R} (|\nabla u_\lambda |^p+a(x)|\nabla u_\lambda |^q) \eta^q \,dx 
		&\le   c \int_{B_R} (|\nabla u_\lambda|^{p-1}+a(x)|\nabla u_\lambda|^{q-1}) \frac{ u_\lambda}{R-r}  \eta^{q-1} \, dx \\
		&  \le \tfrac{1}{2}  \int_{B_R} (|\nabla u_\lambda|^{p}+a(x)|\nabla u_\lambda|^{q})   \eta^{q} \, dx\\
		&\qquad  + c\int_{B_R} \left(\frac{ u_\lambda}{R-r}\right)^p+ a(x) \left(\frac{ u_\lambda}{R-r}\right)^q \, dx,
	\end{align*}
	which implies the desired estimate.
\end{proof}

\begin{remark}\label{rmk:Caccio}
	If $u \in W^{1,\phi}(\Omega)$ is a weak supersolution of \eqref{eq:PDE_full}, then we have the same inequality as in the above lemma  with $u_\lambda$ replaced by $(-u)_\lambda$.
\end{remark}
Now, we prove the local boundedness of the weak solution of \eqref{eq:PDE_full}.
\begin{theorem}[$L^\infty$-estimates] \label{thm:local_bound_with_average} 
	Let $\phi\in\mathcal A$. If $u \in W^{1,\phi}(\Omega)$ is a weak subsolution of \eqref{eq:PDE_full}, then $u$ is locally bounded from above. Moreover, we have that for any ball $B$ with $2B\Subset \Omega$ and with $\int_{2B}\phi(x,|\nabla u|)\,dx \leq 1$,
	\begin{align}\label{eq:upperbound}
		(M_{2B} \phi)\bigg(\frac{\norm{(u -\mean{u}_{2B})_+}_{L^\infty(B)}}{r}\bigg)\lesssim
		\dashint_{2B} \phi\bigg(x,\frac{(u -\mean{u}_{2B})_+}{r}\bigg)\,dx.
	\end{align}
	
	Therefore, if $u \in W^{1,\phi}(\Omega)$ is a weak solution of \eqref{eq:PDE_full}, then $u$ is locally bounded with the following estimate:  for any ball $B$ with $2B\Subset \Omega$ with $\int_{2B}\phi(x,|\nabla u|)\,dx \leq 1$,
	\begin{align*}
		(M_{2B} \phi)\bigg(\frac{\norm{u -\mean{u}_{2B}}_{L^\infty(B)}}{r}\bigg)\lesssim
		\dashint_{2B} \phi\bigg(x,\frac{\abs{u -\mean{u}_{2B}}}{r}\bigg)\,dx.
	\end{align*}
	or equivalently
	\begin{align}\label{eq:supremum}
		\frac{\norm{u -\mean{u}_{2B}}_{L^\infty(B)}}{r} \lesssim (M_{2B}\phi)^{-1}\bigg(\dashint_{2B} \phi\bigg(x,\frac{\abs{u -\mean{u}_{2B}}}{r}\bigg)\,dx\bigg).	
	\end{align}
	
	Note that the implicit constants  depend only on $p,q,[\phi]_\mathcal{A}$ and $n$.
\end{theorem}
\begin{proof}
	Let $r>0$ be the radius of $B$, and  $\lambda_\infty$ be some positive constant which will be determined later. For $k\in \mathbb N\cup\{0\},$ we define $\lambda_k\coloneqq \lambda_{\infty}(1-2^{-k})$, $r_k\coloneqq r(1+2^{-k}),$ $\tilde{r}_{k}\coloneqq\tfrac{1}{2}(r_k+r_{k+1}),$  $B_k\coloneqq B_{r_k},$\,\,$\tilde{B}_k\coloneqq B_{\tilde{r}_k},$	$v_k \coloneqq (u-\mean{u}_{2B})_{\lambda_k},$ 
	$\eta_{k+1}\in C^{\infty}_c(\tilde B_{r_{k}},[0,1]),$ where
	$\eta_{k+1}=1$ on $B_{k+1}$ and $|\nabla \eta_{k+1}| \le \frac{2^{k+2}}{r},$
	and 
	$$
	V_k\coloneqq \dashint_{B_{k}}\phi\Big(x,\frac{v_k}{r_k}\Big)\,dx.
	$$ 
	Using H{\"o}lder's  inequality with $s>1$ we get
	\begin{align}\label{eq:Vk_with_average}
		V_{k+1}\leq C\Bigg(
		\dashint_{B_{k+1}}\phi^s \Big(x,\frac{v_{k+1}}{r_{k+1}}\Big)\,dx\Bigg)^{\frac{1}{s}}
		\Bigg(\frac{|\lbrace v_{k+1}>0 \rbrace\cap B_{k+1}|}{|B_{k+1}|}\Bigg)^{1-\frac{1}{s}}.
	\end{align}	
	Observe that using \eqref{eq:Sobolev}, we get the following \Poincare inequality:
	\begin{align}\label{eq:vk}
		\int_{B_k} \phi\Big( x, \frac{v_k}{r_k} \Big)\, dx
		\lesssim  \int_{2B} \phi\Big( x, \frac{|u-\mean{u}_{2B}|}{2r} \Big)\, dx \lesssim \int_{2B} \phi ( x,|\nabla u| )\, dx \le 1.
	\end{align}
	Then by Theorem~\ref{thm:Jensen} \ref{itm:Jensen_b} with $f=\indicator_{\{v_{k+1}>0\}\cap B_{k+1}}\frac{\lambda_{k+1}-\lambda_{k}}{r_k}$ it follows that 
	\begin{align*}
		\dashint_{B_k} \phi \Big(x, \frac{v_k}{r_k}\Big)\,dx
		&\ge
		\dashint_{B_k}
		\phi \Big(x, \indicator_{\{v_{k+1}>0\}\cap B_{k+1}}\frac{\lambda_{k+1}-\lambda_{k}}{r_k}\Big) 
		\,dx
		\\
		&\gtrsim\dashint_{B_k}\phi \Big(x, \frac{\abs{\{v_{k+1}>0\}\cap B_{k+1}}}{\abs{B_k}}\frac{\lambda_{k+1}-\lambda_{k}}{r_k}\Big)
		\,dx
		\\
		&\eqsim\dashint_{B_k}\phi \Big(x, \frac{\abs{\{v_{k+1}>0\}\cap B_{k+1}}}{\abs{B_{k+1}}}\frac{2^{-(k+1)} \lambda_{\infty}}{r}\Big)
		\,dx,
	\end{align*} 
	which yields
	\begin{align*}
		\frac{\abs{\{v_{k+1}>0\}\cap B_{k+1}}}{\abs{B_{k+1}}}
		\lesssim 2^k \bigg(\frac{r}{\lambda_{\infty}}\bigg)(M_{2B} \phi)^{-1}(V_k).
	\end{align*}
	We next estimate the first factor on the right hand side in \eqref{eq:Vk_with_average}. By the Sobolev-\Poincare inequality \eqref{eq:Sobolev1}, 
	\begin{align*}
		\begin{aligned}
			& \bigg(\dashint_{B_{k+1}}\phi^s\Big(x,\frac{v_{k+1}}{r_{k+1}}\Big)\,dx\bigg)^{\frac{1}{s}}
			\lesssim  2^{qk} \bigg(\dashint_{\tilde{B}_k}\phi^s\Big(x,\frac{2^{-k}v_{k+1}\eta_{k+1}}{\tilde{r}_{k}}\Big)\,dx\bigg)^{\frac{1}{s}}\\
			&\lesssim 2^{qk}\dashint_{\tilde{B}_k}\phi(x,|\nabla(2^{-k}v_{k+1}\eta_{k+1})|)\,dx
			\\
			&\lesssim 2^{qk} \bigg(\dashint_{\tilde{B}_k}\phi(x,2^{-k}v_{k+1}|\nabla\eta_{k+1}|)\, dx +  \dashint_{B_{\tilde r_k}}\phi
			(x,|\nabla v_{k+1}|)\, dx\bigg)
			\\
			&\lesssim 2^{qk}\dashint_{B_k}\phi\Big( x, \frac{v_{k+1}}{r}\Big)\, dx
			\\
			&\lesssim 2^{qk}\dashint_{B_k}\phi\Big( x, \frac{v_{k}}{r_k}\Big)dx= 2^{qk}V_k.
		\end{aligned}
	\end{align*} 
	Note that, when applying \eqref{eq:Sobolev1}, we use the following fact, which follows from \eqref{eq:vk}: 
	\begin{align*}
		&\int_{B_{\tilde r_k}}\phi(x,|\nabla(2^{-k}v_{k+1}\eta_{k+1})|)\,dx\\
		&\lesssim \int_{B_{\tilde r_k}}\phi(x,2^{-k}v_{k+1}|\nabla\eta_{k+1}|)\,dx + \int_{B_{\tilde r_k}}\phi(x,|\nabla v_{k+1}|)\,dx
		\\
		&\lesssim \int_{2B}\phi\Big( x, \frac{|u-\mean{u}_{2B}|}{2r}\Big)\,dx +\int_{2B}\phi(x,|\nabla u |)\, dx
		\\
		&\lesssim \int_{2B}\phi ( x, |\nabla u|)\, dx \le 1 .
	\end{align*}
	Using the inequalities above, after multiplying them we get that 
	\begin{align*}
		V_{k+1}\leq c_0 2^{(1-\frac{1}{s}+q)k}V_k\Big((M_{2B}\phi)^{-1}(V_k)\tfrac{r}{\lambda_{\infty}}\Big)^{1-\frac{1}{s}}
	\end{align*}
	for some $c_0>0$ depending on  $p,q,[\phi]_\mathcal{A}$ and $n$.	
	Finally, as in the proof of the convergence of the De~Giorgi iteration (see, e.g., \cite[Lemma 7.1]{GIUS}), if we choose $\lambda_\infty$ such that
	\begin{align*}
		(M_{2B}\phi)^{-1}(V_0)  =  c_0^{-\frac{1}{\vartheta}} L^{-\frac{q}{\vartheta^2}} \frac{\lambda_\infty}{r}, 
		\quad \text{where}\ \ L:= 2^{1-\frac{1}{s}+q} \ \ \text{and}\ \ \vartheta=1-\frac{1}{s},
	\end{align*}
	one can prove that $V_k \le L^{-\frac{q}{\vartheta}k} V_0$ for all $k\in \mathbb N \cup\{0\}$ and, in particular, $V_k\to 0$ as $k\to\infty$. For completeness, we prove the inequality $V_k \le L^{-\frac{q}{\vartheta}k} V_0$ by induction.
	For the base case $k=0$, the inequality is trivial. 
	For the induction step, suppose the inequality $V_k \le L^{-\frac{q}{\vartheta}k} V_0$ holds for some $k\ge 0$. Then,
	\begin{align*}
		V_{k+1} & \le c_0 L^{k}L^{-\frac{q}{\vartheta}k} V_0 \Big((M_{2B}\phi)^{-1}\big(L^{-\frac{q}{\vartheta}k} V_0\big)\tfrac{r}{\lambda_{\infty}}\Big)^{\vartheta}\\
		& \le c_0 L^{k}L^{-\frac{p}{\vartheta}k} V_0 \Big(L^{-\frac{1}{\vartheta}k} (M_{2B}\phi)^{-1}\big( V_0\big)\tfrac{r}{\lambda_{\infty}}\Big)^{\vartheta}\\
		&=c_0 L^{k}L^{-\frac{q}{\vartheta}k} V_0 \Big(L^{-\frac{1}{\vartheta}k} c_0^{-\frac{1}{\vartheta}} L^{-\frac{q}{\vartheta^2}}  \Big)^{\vartheta} =  L^{-\frac{q}{\vartheta}(k+1)} V_0.
	\end{align*}
	Considering this we conclude that, for almost every $x\in B$, we have
	\begin{equation*}
		(u(x)-\mean{u}_{2B})_+ \le\lambda_\infty = c_0^{\frac{1}{\vartheta}}L^{\frac{p}{\vartheta^2}} r (M_{2B} \phi)^{-1}(V_0),
	\end{equation*}
	which implies \eqref{eq:upperbound}.
\end{proof}

	\begin{corollary}[$L^\infty$-estimates] \label{thm:local_bound} 
		Let $\phi\in\mathcal A$. If $u \in W^{1,\phi}(\Omega)$ is a weak subsolution of \eqref{eq:PDE_full}, and $\tfrac{\abs{E}}{\abs{B}}\ge \delta_0$, where $E\coloneqq \{ x \in 2B : u(x) =0 \}$,  for some $\delta_0\in(0,1)$,
		then we have that for any ball $B$ with $2B\Subset \Omega$ with $\int_{2B}\phi(x,|\nabla u|)\,dx \le 1$,
		\begin{equation}\label{eq:sub_supestimate}\begin{aligned}
				&(M_{2B} \phi)\left( \frac{\| u_+\|_{L^\infty(B)}}{r}\right)
				\le C  \dashint_{2B}\phi\Big(x,\frac{u_+}{r}\Big)\, dx,
		\end{aligned}\end{equation}
		where the constant $C>0$ depends on $p,q,[\phi]_{\mathcal{A}}$ and the ratio $\delta_0$.
	\end{corollary}
	\begin{proof}
		The proof is essentially the same as that of \eqref{eq:upperbound}, except that we replace the definition of $v_k$ by
		$v_k\coloneqq u_{\lambda_k}= (u-\lambda_k)_+$ and, instead of using \eqref{eq:vk}, apply the following \Poincare inequality, 	which follows	from Corollary~\ref{cor:sobolev-poincare-zeroset}:
		\begin{align*}
			\int_{2B}\phi\Big( x, \frac{|u|}{2r}\Big)\,dx \lesssim \int_{2B}\phi ( x, |\nabla u|)\, dx \le 1.
		\end{align*}
\end{proof}

\subsection{H\"older continuity}

We prove the H\"older continuity of the weak solutions of \eqref{eq:PDE_full}.  We first recall the following well-known estimate, which follows from the Sobolev-\Poincare inequality for $W^{1,1}$ functions; see, e.g., \cite[Lemma 3.5]{LU}. Let $B \subset \mathbb{R}^n$ be any ball and let $u \in W^{1,1}(B)$. For any real numbers $\mu$ and $\lambda$ with $\mu <\lambda$, and for any $s > 1$, consider the sets $\mathcal{D} := \{x \in B : u(x) > \lambda\}$, $\mathcal{E} := \{x \in B : u(x) \le \mu\}$, and $\mathcal{F} := \{x \in B : \mu < u(x) \le \lambda \}$. Then, the following inequality holds:
\begin{equation}\label{poincaretype1}
	(\lambda-\mu)|\mathcal D|^{\frac{n-1}{n}} \le C_n \frac{|B|}{|\mathcal E|}\int_{\mathcal F}|\nabla u|\,dx \le C_n\frac{|B|}{|\mathcal E|}\left(\int_{\mathcal F}|\nabla u|^{s}\,dx\right)^{\frac{1}{s}} |\mathcal F|^{1-\frac{1}{s}},
\end{equation}
where $C_n>0$ is a constant depending only on $n$.
The following elementary result will be used in the proof of the next lemma.	
\begin{lemma}\label{lem:sequence}
	Let $\{a_j\}_{j=1}^\infty$ be a sequence of nonnegative real numbers such that $\sum_{j=1}^\infty a_j \le 1$. 
	Suppose that there exist constants $c_0, \beta_0 > 0$ satisfying
	$\sum_{j=\ell+1}^\infty a_j \le c_0\, a_{\ell}^{\beta_0}$ for all $\ell \in \mathbb{N}$.
	Then, we have that
	$\sum_{j=2\ell}^\infty a_j \le c_0\, \ell^{-\beta_0}$ for all $\ell \in \mathbb{N}$.
\end{lemma}	
\begin{proof}
	Fix $\ell\in \mathbb N$. Since $\sum_{j=1}^\infty a_j\le 1$,  there exists $\ell_1\in \{\ell,\ell+1,\dots, 2\ell-1\}$ such that $a_{\ell_1}\le \frac{1}{\ell}$. Therefore, $\sum_{j= 2\ell}^\infty a_j \le \sum_{j= \ell_1+1}^\infty a_j\le  c_0 a_{\ell_1}^{\beta_0}\le  c_0 \ell^{-\beta_0}$
\end{proof}

We prove the following density improvement lemma, which will imply the H\"{o}lder continuity.

\begin{lemma}\label{lem:density0}
	Let $w \in W^{1,\phi}(\Omega)$ be a weak supersolution of \eqref{eq:PDE_full} and $\nu >0$.  For any  $\sigma\in(0,1]$, there exists a constant $\epsilon\in (0,\frac12)$ depending on $n,$ $p,$ $q,$ $[\phi]_{\mathcal A}$ and $\sigma$ (independent of $\nu$) such that if   $\int_{4B}\phi(x,|\nabla w|)\, dx \le 1$ with $4B\Subset\Omega$,  $w$ is nonnegative in $4B$, and 
	\begin{equation}\label{0densitycondition}
		\frac{|\{x\in 2B : w(x) \ge \nu \}|}{|2B|} \ge \sigma,
	\end{equation}
	then
	\begin{equation}\label{0positivity}
		\inf_{B} w \ge \epsilon \nu.
	\end{equation}
\end{lemma}

\begin{proof} 
	We divide the proof into two steps.
	
	\textit{Step 1.}
	We first prove the following claim. 
	
	{\bf Claim.} \textit{For every $\tau>0$ there exists a small $\epsilon\in(0,\frac14)$ depending on $n,$ $p,$ $q,$ $[\phi]_{\mathcal A},$ $\sigma$ and $\tau$ such that 
	\begin{equation}\label{0density}
		\frac{|\{x \in 2B : w(x) \le 2\epsilon\nu \}|}{|2B|} \le \tau.
	\end{equation}}
	
	Let $r>0$ be the radius of $B$. For $j\in \mathbb N$, define $k_j=2^{1-j}\nu$ and $\mathcal F_j := \{x\in 2B :  k_{j+1} < w(x) \le k_j \}$. Note that
	\begin{equation}\label{0levelset_estimates}
		|\{w\ge k_j\}\cap 2B | \ge |\{w\ge \nu \}\cap 2B| \geq \sigma|2B| \quad \text{for all }\ j\in\mathbb N.
	\end{equation}
	Since $\int_{4B} \phi(x, |\nabla w|)\,dx \le 1,$
	we use Lemma  \ref{lem:Improved_inside} to obtain
	\begin{align}
		\dashint_{4B} \phi\!\left(
		x, \left(\fint_{4B} \indicator_{\mathcal F_j} |\nabla w|^{s}\,dy\right)^{\frac{1}{s}}
		\right)\,dx
		&\lesssim \dashint_{4B}  \phi\!\left(x,  \indicator_{\mathcal F_j}|\nabla w| \right)\,dx \notag\\
		&\lesssim \fint_{2B} \phi\!\left(x, |\nabla (k_j-w)_+| \right)\,dx
	\end{align}
	for some $s>1$.
Applying the Caccioppoli estimate (Lemma~\ref{lem:Caccio}) with the choice $u_\lambda = (-w)_{-k_j} = (k_j-w)_+$ (see also Remark~\ref{rmk:Caccio}), and using the nonnegativity of $w$, we obtain
	$$\begin{aligned}
		\fint_{2B}  \phi\left(x, |\nabla (k_j-w)_+| \right) \,dx
		\lesssim \fint_{4B} \phi\left(x, \frac{(k_j-w)_+}{r} \right) \,dx 
		\lesssim \fint_{4B} \phi\left(x, \frac{k_j}{r} \right) \,dx.
	\end{aligned}$$
	Combining the above two estimates and the monotonicity of $M_{4B}\phi$ yields that 
	\[
	\left(\fint_{2B}\indicator_{\mathcal F_j}|\nabla w|^{s}\,dy\right)^{\frac{1}{s}} \lesssim \frac{k_j}{r} = \frac{2^{1-j}\nu }{r}.
	\]
Hence, applying \eqref{poincaretype1} (with $B$ replaced by $2B$, $\lambda=k_j$, and $\mu=k_{j+1}$) combined with \eqref{0levelset_estimates} and the previous inequality, we obtain that
	\[\begin{split}
		2^{-j}\nu \sigma^{\frac{n-1}{n}} \frac{| 2B \cap\{  w \le k_{j+1}\}|}{|2B|}
		&\le (k_j-k_{j+1}) \frac{| 2B \cap\{  w \le k_{j+1}\}|}{|2B|}\left(\frac{|2B \cap \{ w > k_j\}|}{|2B|}\right)^{\frac{n-1}{n}}\\ 
		&\lesssim r   \left( \fint_{2B}\indicator_{\mathcal F_j}|\nabla w|^{s}\,dy\right)^{\frac{1}{s}} \left(\frac{|\mathcal F_j|}{|2B|}\right)^{1-\frac{1}{s}}\\
		&\lesssim  2^{-j}\nu \left(\frac{|\mathcal F_j|}{|2B|}\right)^{1-\frac{1}{s}},
	\end{split}\]
	which implies that
	\begin{align*}
		\frac{| 2B \cap\{  w \le k_{j+1}\}|}{|2B|}  \lesssim \sigma^{\frac{1-n}{n}}\left(\frac{|\mathcal F_j|}{|2B|}\right)^{1-\frac{1}{s}}.
	\end{align*}
This implies that $\abs{ \{ x \in 2B : w(x) = 0 \} } = 0$, and consequently we obtain
	\begin{align}\label{eq:estimate-for-lemma3.5}
		\sum_{i=j+1}^{\infty}\frac{\abs{\mathcal{F}_i}}{\abs{2B}}=\frac{| 2B \cap\{  w \le k_{j+1}\}|}{|2B|} \le c_1 \sigma^{\frac{1-n}{n}}\left(\frac{|\mathcal F_j|}{|2B|}\right)^{1-\frac{1}{s}}
	\end{align} 
	for some $c_1>0$ depending on $n$, $p$, $q$ and $[\phi]_{\mathcal A}$. Considering \eqref{eq:estimate-for-lemma3.5}, the fact that  $\sum_j |\mathcal{F}_j| \le |2B|$, and applying Lemma~\ref{lem:sequence} to $a_j= |\mathcal{F}_j| / |2B|$, $c_0=c_1\sigma^{\frac{1-n}{n}}$ and $\beta_0=1-\frac{1}{s}$, we obtain that
	\begin{align*}
		\frac{| 2B \cap\{  w \le k_{2\ell}\}|}{|2B|}  \le c_ 1 \sigma^{\frac{1-n}{n}}\ell^{-(1-\frac{1}{s})}
		\quad \text{for all }\ \ell\in\mathbb N.
	\end{align*}
	This completes the proof of the claim by choosing $\ell\in\mathbb N$ such that $c_1  \sigma^{\frac{1-n}{n}}\ell^{-(1-\frac{1}{s})}\le \tau$ and then
	$\epsilon=2^{-2\ell+1}$.
	
	\smallskip
	
	\textit{Step 2.}  We now complete the proof of the lemma. Assume that $\tau \in (0,\frac{1}{2})$.
	 Let $v:=(2\epsilon \nu -w)_+$. Then $v$ is a weak subsolution to \eqref{eq:PDE_full} with $\int_{4B}\phi(x,|\nabla v|)\,dx \le \int_{4B}\phi(x,|\nabla w|)\,dx\le 1$, and 
		it follows from \eqref{0density} that $|\{v=0\}\cap 2B| =  |\{w\ge 2\epsilon \nu \}\cap 2B| \ge (1- \tau)|2B|\ge \frac{1}{2}|2B|$. 
		Applying \eqref{eq:sub_supestimate} to $v$, combined with the reverse H\"older inequality \eqref{eq:phi_RH} (with $t=\frac{\|v\|_{L^\infty(2B)}}{r}$), \eqref{0density}, and the nonnegativity of $w$, we obtain
		\begin{align*}
			M_{2B}\phi \left(\frac{\|v\|_{L^\infty(B)}}{r}\right) 
			& \lesssim \dashint_{2B} \phi\left(x, \frac{v}{r}\right)\,dx\\
			& \lesssim \left(\frac{|\{v >0  \}\cap 2B|}{|2B|}\right)^{\frac{s-1}{s}} \left(\dashint_{2B}\phi^s\left(x,  \frac{\|v\|_{L^\infty(2B)}}{r}\right)\,dx\right)^{\frac{1}{s}} \\
			& \lesssim \left(\frac{|\{w < 2\epsilon \nu  \}\cap 2B|}{|2B|}\right)^{\frac{s-1}{s}}  \dashint_{2B}\phi\left(x,  \frac{\|v\|_{L^\infty(2B)}}{r}\right)\,dx  \\
			& \lesssim \tau^{\frac{s-1}{s}} M_{2B}\phi\left( \frac{\epsilon \nu}{r}\right).
		\end{align*}
		for some $s>1$ depending on $n$, $p$, $q$, and $[\phi]_{\mathcal A}$. To justify the application of \eqref{eq:phi_RH} and \eqref{eq:sub_supestimate}, we observe two facts. First, the zero set of $v$ satisfies
		$|\{v=0\}\cap 4B| \ge |\{w\ge 2\epsilon \nu \}\cap 2B| \ge (1- \tau)|2B|\ge \frac{1}{2^{n+1}}|4B|.$
		Second, we justify the applicability of \eqref{eq:sub_supestimate} by using the Poincaré inequality as follows
		\begin{align*}
			\int_{2B}\phi\left(x,  \frac{\|v\|_{L^\infty(2B)}}{r}\right)\,dx
			\lesssim   \int_{4B}\phi\left(x,  \frac{v}{r}\right)\,dx 
			&\lesssim \int_{4B}\phi (x,  |\nabla v|)\,dx \\
			&\le \int_{4B}\phi (x,  |\nabla w|)\,dx\le 1.
		\end{align*}
		Choosing $\tau \in (0, \frac{1}{2})$ small enough such that
		\begin{align*}
			M_{2B}\phi \left(\frac{\|v\|_{L^\infty(B)}}{r}\right) \le M_{2B}\phi\left( \frac{\epsilon \nu}{r}\right),
		\end{align*}
		and considering the monotonicity of $M_{2B}\phi$, we obtain $\|(2\epsilon \nu -w)_+\|_{L^{\infty}(B)} = \|v\|_{L^{\infty}(B)}\le \epsilon \nu$. This completes the proof.
\end{proof}

Now, we are ready to prove the H\"older continuity of the weak solution of \eqref{eq:PDE_full}, which establishes Theorem~\ref{thm:local_Holder}.

\begin{proof}[Proof of Theorem~\ref{thm:local_Holder}]
	
    We have already shown that the weak solution $u$ is locally bounded with the estimate \eqref{eq:supremum}. Therefore, it is enough to show that the following oscillation decay estimate holds
	\begin{align}\label{eq:oscillation_decay}
		\sup_{B}u - \inf_{B}u \le  \theta (\sup_{4B} u - \inf_{4B} u),
	\end{align} 
	for any $4B\Subset \Omega$ with $\int_{4B}\phi(x,|\nabla u|)\, dx\le 1$, for some $\theta\in(0,1)$ depending only on $n,$ $p,$ $q,$ and $[\phi]_{\mathcal A}$. Then by iterating \eqref{eq:oscillation_decay} for  shrinking balls we obtain the boundedness of the $\beta$-H\"older seminorm of $u$ for some $\beta\in(0,1)$ depending only on $n,$ $p,$ $q,$ and $[\phi]_{\mathcal A}$.
	
	In order to prove \eqref{eq:oscillation_decay}, we first observe that either
	$$
	\frac{\left|\left\{u- \inf_{4B} u \ge  \tfrac{1}{2}(\sup_{4B} u - \inf_{4B} u)\right\}\cap 2B \right|}{|2B|} \ge \frac12  
	$$
	or
	$$
	\frac{\left|\left\{\sup_{4B} u -u \ge  \tfrac{1}{2}(\sup_{4B} u - \inf_{4B} u)\right\}\cap 2B \right|}{|2B|} \ge \frac12 
	$$
	holds true.
	In the first case, by Lemma~\ref{lem:density0} with $w=u-\inf_{4B}u$ and $\nu=\frac{1}{2}(\sup_{4B} u - \inf_{4B} u)$, we have 
	$$
	\inf_{B}u- \inf_{4B} u \ge \frac{\epsilon}{2}\Big(\sup_{4B} u - \inf_{4B} u\Big),
	$$
	which implies that
	$$
	\sup_{B}u - \inf_{B}u \le  \sup_{B}u - \inf_{4B} u - \frac{\epsilon}{2}\Big(\sup_{4B} u - \inf_{4B} u\Big).
	$$
	Therefore, since $\sup_{B}u\le \sup_{4B}u$, we obtain \eqref{eq:oscillation_decay} with $\theta=1-\frac{\epsilon}{2}$.	
	On the other hand, in the second case, again by Lemma~\ref{lem:density0} with $w=\sup_{4B} u - u$ and $\nu=\frac{1}{2}(\sup_{4B} u - \inf_{4B} u)$, we have 
	$$
	\sup_{4B} u - \sup_{B}u  \ge \frac{\epsilon}{2}\Big(\sup_{4B} u - \inf_{4B} u\Big),
	$$
	which implies that
	$$
	\sup_{B}u - \inf_{B}u \le  \sup_{4B}u - \inf_{B} u - \frac{\epsilon}{2}\Big(\sup_{4B} u - \inf_{4B} u\Big).
	$$
	Therefore, since $\inf_{B}u\ge \inf_{4B}u$, we again obtain \eqref{eq:oscillation_decay} with $\theta=1-\frac{\epsilon}{2}$. Therefore, the proof is completed.
\end{proof} 

\begin{remark}[More general models]
  Our proof of $C^{0,\alpha}$-regularity of the solution relies on several properties like \Poincare{} and Sobolev-\Poincare{} inequalities in different forms. However, all of these properties are deduced from the generalized Jensen's inequality Theorem~\ref{thm:Jensen}. Hence, the results of our paper generalize to all models based on a generalized Young function that allow to get Theorem~\ref{thm:Jensen}. A well known example is certainly the Young function $t^p \omega(x)$ with $\omega \in A_p$ from \cite{FKS}. But we expect that more models satisfy Theorem~\ref{thm:Jensen}.
\end{remark}

\printbibliography
\end{document}